\documentclass[11pt]{article}
\usepackage[ansinew]{inputenc}
\usepackage{graphicx}
\usepackage{color}
\usepackage{times}

\usepackage{enumerate,latexsym}
\usepackage{latexsym}
\usepackage{amsmath,amssymb}
\usepackage{graphicx}
\newfont{\bb}{msbm10 at 11pt}
\newfont{\bbsmall}{msbm8 at 8pt}

\def\rth{\mathbb{R}^3}
\def\R{\mathbb{R}}
\def\B{\mathbb{B}}
\def\N{\mathbb{N}}
\def\Z{\mathbb{Z}}
\def\C{\mathbb{C}}

\def\esf{\mathbb{S}}

\newcommand{\ben}{\begin{enumerate}}
\newcommand{\bit}{\begin{itemize}}
\newcommand{\een}{\end{enumerate}}
\newcommand{\eit}{\end{itemize}}
\newcommand{\wh}{\widehat}
\newcommand{\Int}{\mbox{Int}}

\newcommand{\wt}{\widetilde}

\newcommand{\h}{\widehat}
\newcommand{\ed}{\end{document}}

\def\a{{\alpha}}

\def\t{{\theta}}

\def\g{{\gamma}}
\def\G{{\Gamma}}
\def\l{{\lambda}}
\def\L{{\Lambda}}
\def\de{{\delta}}
\def\b{{\beta}}
\def\ve{{\varepsilon}}

\def\centerbmp#1#2#3{\vskip#2\relax\centerline{\hbox to#1{\special
    {bmp:#3 x=#1, y=#2}\hfil}}}

\newtheorem{theorem}{Theorem}[section]
\newtheorem{lemma}[theorem]{Lemma}
\newtheorem{proposition}[theorem]{Proposition}

\newtheorem{remark}[theorem]{Remark}
\newtheorem{corollary}[theorem]{Corollary}
\newtheorem{definition}[theorem]{Definition}

\newtheorem{assertion}[theorem]{Assertion}

\newenvironment{proof}{\smallskip\noindent{\it Proof.}\hskip \labelsep}
{\hfill\penalty10000\raisebox{-.09em}{$\Box$}\par\medskip}

\begin{document}
\begin{title}
{Embedded minimal surfaces of finite topology}
\end{title}
\vskip .2in

\begin{author}
{William H. Meeks III\thanks{This material is based upon
   work for the NSF under Award No. DMS -
  1309236. Any opinions, findings, and conclusions or recommendations
   expressed in this publication are those of the authors and do not
   necessarily reflect the views of the NSF.}
   \and Joaqu\'\i n P\' erez
\thanks{Research partially supported by the MINECO/FEDER grant no. MTM2011-22547.}, }
\end{author}
\maketitle
\begin{abstract}
In this paper we prove that a complete, embedded minimal surface
$M$ in $\rth$ with finite topology and compact boundary (possibly
empty) is conformally a compact Riemann surface $\overline{M}$ with
boundary punctured in a finite number of interior points and that
$M$ can be represented in terms of meromorphic data on its conformal
completion $\overline{M}$. In particular, we demonstrate that $M$ is
a minimal surface of finite type and describe how this property
permits a classification of the asymptotic behavior of $M$.

\vspace{.3cm}


\end{abstract}

\section{Introduction.} \label{sec1}

Based on work of Colding and Minicozzi~\cite{cm23},
Meeks and Rosenberg~\cite{mr8} proved that a properly embedded,
simply connected minimal surface in $\rth$ is a plane or a helicoid.
In the last page of their paper, they described how their proof
of the uniqueness of the helicoid could be modified  to prove:
\begin{quote}
{\em  Any nonplanar, properly embedded minimal surface $M$ in $\rth$
with one end, finite topology and infinite total curvature satisfies
the following properties:
\begin{enumerate}[1.]
\item $M$ is conformally a compact Riemann surface
$\overline{M}$ punctured at a single point.
\item $M$ is asymptotic to a helicoid.
\item $M$ can be expressed analytically in terms of
meromorphic data on $\overline{M}$.
\end{enumerate}
}
\end{quote}

A rigorous proof of the above statement has been given
by Bernstein and Breiner~\cite{bb2}.

In our survey~\cite{mpe2} and book~\cite{mpe10}, we outlined the proof by Meeks and Rosenberg
of the uniqueness of the helicoid and at the end of this outline we
mentioned how some difficult parts of this proof could be simplified
using some results of Colding and Minicozzi in~\cite{cm34}, and
referred the reader to the present paper for details. Actually, we will
 consider the more general problem of
describing the asymptotic behavior, conformal structure and analytic
representation of an annular end of any {\it complete,} injectively
immersed\footnote{That is, $M$ has no self-intersections and its intrinsic
 topology may or may not agree with the subspace topology as a subset of
 $\R^3$.}
minimal surface $M$ in $\rth$ with compact boundary
and finite topology.
Although not explicitly stated in the paper~\cite{cm35} by
Colding and Minicozzi, the results contained there imply
that such an $M$ is {\it properly embedded}\/\footnote{By this we mean that
the intrinsic topology on $M$ coincides with the subspace topology and
the intersection of $M$ with every closed ball in $\R^3$ is compact
in $M$.} in $\rth$; we will use this
properness property of $M$ in the proof of Theorem~\ref{th1.1} below.
We also remark that Meeks, P\'erez and Ros~\cite{mpr9} have proved the following
more general properness result: If $M$ is a complete, injectively
immersed minimal surface of finite genus, compact boundary and
a countable number of ends in $\R^3$, then $M$ is proper
(this follows from part 3-B of Theorem~1.3 in~\cite{mpr9}).
Since properly embedded minimal annuli $E\subset \rth$ with compact
boundary and finite total curvature are conformally punctured disks,
are asymptotic to the ends of planes and catenoids and have a well-understood
analytic description in terms of meromorphic
data on their conformal completion, we will focus our attention on
the case that the minimal annulus has infinite total curvature.

\begin{theorem}
\label{th1.1}
Let $E\subset \rth$ be a complete, embedded minimal annulus with
infinite total curvature and compact boundary.
Then, the following properties hold:
\begin{enumerate}[1.]
\item $E$ is properly embedded in $\R^3$.
\item $E$ is conformally diffeomorphic to a punctured disk.
\item After applying a suitable homothety and rigid motion to a subend of $E$, then:
\begin{enumerate}[a.]
\item The holomorphic height differential $dh=dx_3+idx_3^*$ of $E$
is $dh=(1+\frac{\l }{z-\mu })\, dz$, defined on $D(\infty ,R)=\{z\in
\C\mid  |z|\geq R \}$ for some $R>0$, where $\l \geq 0$ and $\mu
\in \C$. In particular, $dh$ extends meromorphically across
 infinity with a double pole.
\item The stereographic projection $g\colon D(\infty ,R)\to \C\cup\{\infty\}$
of the Gauss map of $E$ can be expressed as $g(z)=e^{iz+f(z)}$ for some
holomorphic function $f$ in $D(\infty ,R)$ with $f(\infty )=0$.
\item $E$ is asymptotic to the end of a helicoid if and only if it
has zero flux (in particular, the number $\l $ in the previous item
3a vanishes in this case).
\end{enumerate}
\end{enumerate}
\end{theorem}

Since complete embedded minimal surfaces with finite topology and empty boundary are
properly embedded in $\R^3$\cite{cm35}, and properly embedded minimal
surfaces with more than one end and finite topology have finite
total curvature~\cite{col1}, then we have the following immediate
corollary (which also appears in~\cite{bb2}). Note also that the zero flux condition in
item~3c of Theorem~\ref{th1.1} holds by Cauchy's theorem when the surface has just
one end.
\begin{corollary}
\label{c1.2} Suppose that $M\subset \rth$ is a complete, embedded
minimal surface with finite topology and empty boundary. Then, $M$ is conformally a
compact Riemann surface $\overline{M}$ punctured in a finite number
of points and $M$ can be described analytically in terms of
meromorphic data $\frac{dg}{g},dh$ on $\overline{M}$, where
$g$ and $dh$ denote the stereographically projected Gauss map and
the height differential of $M$, after a suitable rotation in $\R^3$.
Furthermore, $M$ has bounded
Gaussian curvature and each of its ends is asymptotic in the $C^k$-topology 
(for any $k\in \N$) to a plane or a half-catenoid (if $M$ has
finite total curvature) or to a helicoid (if $M$ has infinite total
curvature).
\end{corollary}

In order to state the next result, we need the following notation.
Given $R,h>0$, let
\begin{align}
C(R)=\{ (x_1,x_2,x_3)\mid x_1^2+x_2^2\leq R^2\} ,\label{eq:cyl}
\\
C(R,h)=C(R)\cap \{ |x_3| \leq h\} , \label{eq:cylcap}
\\
D(\infty ,R)=\{ x_3=0\} \cap
[\R^3-\mbox{Int}(C(R))]. \nonumber
\end{align}
A {\it multivalued graph} over $D(\infty ,R)$ is the graph $\Sigma =\{ (re^{i\theta
},u(r,\theta ))\} \subset \C \times \R \equiv \R^3$ of a smooth function
$u=u(r,\theta )$ defined on the universal cover
$\widetilde{D}(\infty ,R)=\{ (r,\theta ) \mid r\geq R,\ \theta \in
\R \} $ of $D(\infty ,R)$.

The next descriptive result is a consequence of Theorem~\ref{main} below, and  provides
 the existence of a  collection ${\cal E}=\{E_{a,b}\mid {a,b\in [0,\infty)}\}$
of properly embedded minimal annuli in $\rth$ such that
any complete,
embedded minimal annulus $E$ with compact boundary and infinite total curvature
in $\R^3$ is, after a rigid motion, asymptotic to a homothetic scaling
of exactly one of the surfaces in ${\cal E}$. The image in
Figure~\ref{Eab} describes how the flux vector
$$(a,0,-b)=\int _{\partial E_{a,b}} \eta $$ of the minimal annulus
$E_{a,b}$  influences its geometry (here $\eta $ refers to the unit conormal
vector exterior to $E_{a,b}$ along its boundary).

\vspace{-.5cm}
\begin{figure}
\begin{center}
\includegraphics[height=7.7cm]{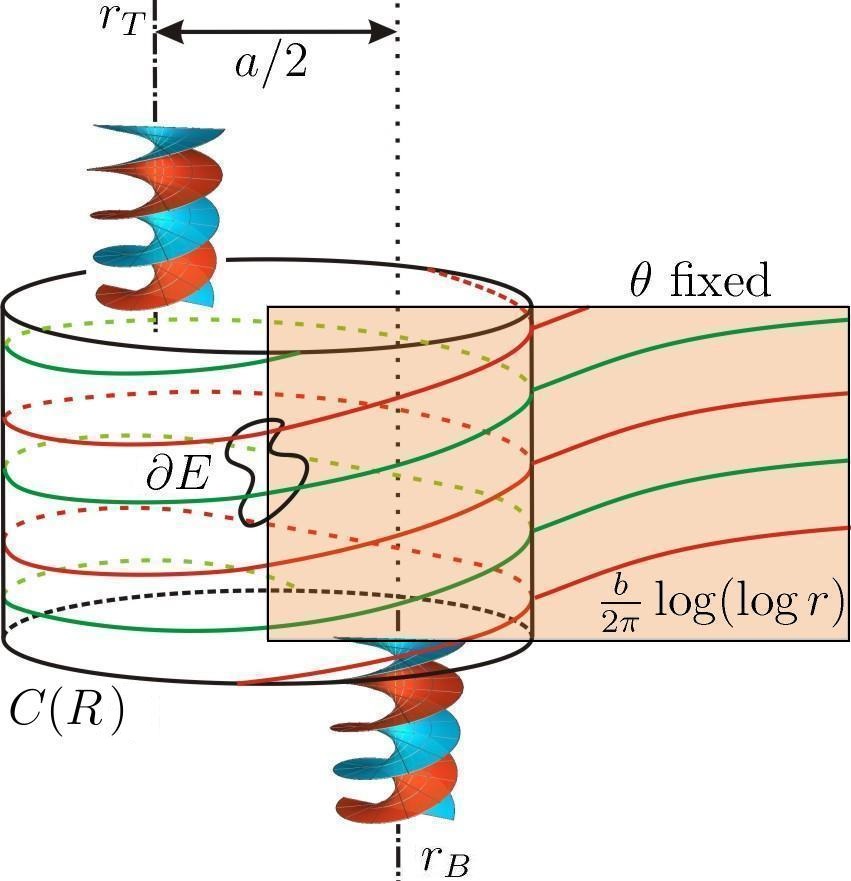}
\caption{The embedded annulus $E=E_{a,b}$ with flux vector $(a,0,-b)$ (see
Theorem~\ref{thm1.3}) has the following description.
Outside the infinite vertical cylinder $C(R)$, $E$ consists of two horizontal
multivalued graphs with asymptotic spacing $\pi $ between them.
The translated surfaces $E-(0,0,2\pi n+\frac{b}{2\pi }\log n)$
(resp. $E+(0,0,2\pi n-\frac{b}{2\pi }\log n)$) converge as $n\to
\infty $ to a vertical helicoid $H_T$ (resp. $H_B$) such that
$H_B=H_T+(0,a/2,0)$ (in the picture, $r_T,r_B$ refer to the axes of
$H_T,H_B$). The intersection of $E-C(R)$ with a vertical half-plane
bounded by the $x_3$-axis consists of an infinite number of curves,
each being a graph of a function $u(r)$ that satisfies the property
$\frac{u(r)}{\log (\log r)}$ converges to $\frac{b}{2\pi }$ as the
radial distance $r$ to the $x_3$-axis tends to~$\infty $.
} \label{Eab}
\end{center}
\end{figure}

\vspace{.4cm}
\begin{theorem}[Asymptotics of embedded minimal annular ends]
\label{thm1.3} Given $a,b\geq 0$, there exist $R>0$
and a properly embedded minimal annulus $E_{a,b}\subset
\R^3$ with compact boundary and flux vector $(a,0,-b)$ along its
boundary, such that the following statements hold.
\begin{enumerate}[1.]
\item $E_{a,b}-C(R)$ consists of two disjoint multivalued graphs $\Sigma _1,
\Sigma _2$ over $D(\infty ,R)$ of functions $u_1,u_2\colon
\widetilde{D}(\infty ,R)\to \R $ such that their gradients, with respect to  the metric
on $\widetilde{D}(\infty ,R)$ obtained by pulling back the standard flat metric 
in $D(\infty ,R)$, satisfy
$\nabla u_i(r,\theta )\to 0$ as $r\to \infty $, and the separation
function $w(r,\theta )=u_1(r,\theta )-u_2(r,\theta )$ between both
multivalued graphs converges to $\pi $ as $r+|\theta |\to \infty $.
Furthermore for $\t $ fixed and $i=1,2$,
\begin{equation}
\label{eq:graphr}
 \lim_{r\to \infty} \frac{u_i(r,\theta)}{\log
(\log(r))} =\frac{b}{2\pi }.
\end{equation}
\item The translated surfaces $E_{a,b}-(0,0,2\pi n+\frac{b}{2\pi
}\log n)$ (resp. $E_{a,b}+(0,0,2\pi n-\frac{b}{2\pi }\log n)$)
converge as $n\to \infty $ to a right-handed vertical helicoid $H_T$  (resp.
$H_B$) with maximum absolute Gaussian curvature 1 along its axis such that
\begin{equation}
\label{eq:HB-HT}
H_B=H_T+(0,a/2,0).
\end{equation}
Note that equations (\ref{eq:graphr}), (\ref{eq:HB-HT})
imply that for different values of $a,b$, the
related surfaces $E_{a,b}$ are not asymptotic after a rigid motion
and homothety.
\item The annuli $E_{0,b}$, $b\in \R $,
are each  invariant under reflection across the $x_3$-axis $l$,
and $l\cap E_{0,b}$ contains  two infinite rays. Furthermore, $E_{0,0}$ is
the end of a vertical helicoid.
\item Up to a rotation and homothety by some $\rho \in \R -\{0\}$, every complete,
embedded minimal annulus in $\R^3$ with compact boundary
and infinite total curvature is asymptotic
to exactly one of the surfaces $E_{a,b}$.
\end{enumerate}
\end{theorem}

As mentioned above, Bernstein and Breiner~\cite{bb2} have given a proof of
Corollary~\ref{c1.2} in the infinite total curvature setting, which
is just Theorem~\ref{th1.1} in the special case that the annulus $E$
is the end of a complete, embedded minimal surface $M$ with finite
topology, one end and empty boundary. Our proof of
Theorem~\ref{th1.1} and the proof of Bernstein and Breiner in the
above special case, apply arguments that Meeks and
Rosenberg~\cite{mr8} used to understand the asymptotic behavior of
an $E$ with zero flux, and to show that such an $E$ has finite
type\footnote{See Definition~\ref{deffinitetype} for the notion of a
minimal surface of finite type.} (even in the case with nonzero flux, this finite type
property will follow from application of Corollary~1.2 in our previous paper~\cite{mpe5}).
Once $E$ is proven to have finite
type, then Meeks and Rosenberg, as well as Bernstein and Breiner,
appeal to results of Hauswirth, P\'erez and
Romon~\cite{hkp1} on the geometry of embedded, minimal annular ends
of finite type to complete the proof. In our proof of
Theorem~\ref{thm1.3}, we essentially avoid reference to the results
in~\cite{hkp1} by using self-contained arguments, some of which have
a more geometric, less analytic nature than those in~\cite{hkp1}.
Central to most of these proofs are the results of Colding and
Minicozzi \cite{cm34,cm22,cm24,cm23,cm35,cm25}.

We refer the interested reader to our paper~\cite{mpr9} with
Ros on the embedded Calabi-Yau problem. More precisely, this paper explains
the structure of simple limit ends of complete, injectively
immersed minimal surfaces $M\subset \rth$ of finite genus and
compact boundary, which together with Theorems~\ref{th1.1}
and~\ref{thm1.3} in this paper, improve our understanding of the
conformal structure, asymptotic behavior and analytic description of
the ends of properly embedded minimal surfaces in $\rth$ that have
compact boundary and finite genus. An outstanding open question
related to the embedded Calabi-Yau problem asks whether every
complete, injectively immersed minimal surface in $\rth$ with finite
genus, compact boundary and an infinite number of ends, must have
only a countable number of ends, since such minimal surfaces are
properly embedded in $\rth$ with exactly one or two limit
ends; see~\cite{mpr9} for further discussion and partial results on
this open problem.

\section{Weierstrass representation of a minimal surface of finite type.}
\label{sec2}
Recall that the Gauss map  of an immersed minimal surface $M$ in $\R^3$
can be viewed as a meromorphic function $g\colon M\to \C \cup \{ \infty \} $
on the underlying Riemann surface, after stereographic projection
from the north pole of the unit sphere. Furthermore, the harmonicity of the
third coordinate function $x_3$ of $M$ allows us to define (at least locally)
its harmonic conjugate function $x_3^*$; hence, the {\it
height differential\/}\footnote{The height differential
 might not be exact since
$x_3^*$ need not be globally well-defined on~$M$. Nevertheless,
the notation $dh$ is commonly accepted and we will also make use of
 it here. }\;  $dh=dx_3+idx_3^*$
is a holomorphic differential on~$M$. As usual, we will refer to the
pair $(g,dh)$ as the {\it Weierstrass data} of $M$,
and the minimal immersion $X\colon M\to \R^3$ can be expressed up to
translation by $X(p_0)$, $p_0\in M$, by
\begin{equation}
\label{eq:repW}
 X(p)=\mbox{Re}\int _{p_0}^p\left( \frac{1}{2}\left(
 \frac{1}{g}-g\right)
,\frac{i}{2}\left( \frac{1}{g}+g\right) ,1\right) dh.
\end{equation}
It is well-known (see Osserman~\cite{os3} for details) that this procedure has a
converse, namely if $M$ is a Riemann surface, $g\colon M\to \C \cup
\{ \infty \} $ is a meromorphic function and $dh$ is a holomorphic
one-form on $M$, such that the zeros of $dh$ coincide with the poles
and zeros of $g$ with the same order, and for any closed curve $\g
\subset M$, the following equalities hold:
\begin{equation}
\label{eq:periodproblem}
    \overline{\int _{\g }g\, dh}=\int _{\g }\frac{dh}{g},\qquad
    \mbox{Re} \int _{\g }dh=0,
\end{equation}
(where $\overline{z}$ denotes the complex conjugate of $z \in \C$), then
the map $X\colon M\to \R^3$ given by {\rm (\ref{eq:repW})} is
a conformal minimal immersion with Weierstrass data $(g,dh)$.

Minimal surfaces $M$ which satisfy the property of having {\it finite type,}
as described in the next definition, can be defined analytically from data
obtained by integrating certain meromorphic one-forms on $M$, which moreover
extend meromorphically to the conformal completion of $M$.

\begin{definition}[Finite Type]
\label{deffinitetype}
{\rm A minimal immersion $X\colon M \to \rth$ is said to have
{\em finite type} if it satisfies the following two properties.
\begin{enumerate}[1.]
\item The underlying Riemann surface to $M$ is conformally diffeomorphic
to a compact Riemann surface $\overline{M}$ with (possibly empty)
compact boundary, punctured in a finite nonempty set ${\cal
E}\subset \Int(\overline{M})$.
\item Given an end $e\in {\cal E}$ of $M$, there exists
a rotation of the surface in $\rth$
such that if $(g,dh)$ is the Weierstrass data of $M$ after this rotation,
then $\frac{dg}{g}$ and $dh$ extend across the puncture~$e$ to meromorphic
one-forms on a neighborhood of $e$ in $\overline{M}$.
\end{enumerate}
}
\end{definition}

Therefore, complete immersed minimal surfaces with finite total
curvature are of finite type, and the basic example of a surface
with finite type but infinite total curvature is the helicoid ($M=\C
$, $g(z)=e^{iz}$, $dh=dz$). Finite type minimal surfaces were first
studied by Rosenberg~\cite{rose1}; we note that his definition of finite type
differs somewhat from our definition above. An in-depth study of {\it
embedded} minimal annular ends of finite type with exact height
differentials was made by Hauswirth, P\'erez and Romon
in~\cite{hkp1}. Some of the results in this paper extend the results
in~\cite{hkp1} to the case where the height differential of the surface might
possibly be not exact.

\section{The conformal type, height differential and Gauss map
of embedded minimal annuli.} \label{sec3} Throughout this section,
we will fix a complete, injectively immersed minimal surface $M\subset\rth$ with
infinite total curvature, where $M$ is an annulus diffeomorphic to
$\esf^1 \times [0,\infty)$. By the main result in Colding and
Minicozzi~\cite{cm35} (which holds in this setting of compact
boundary), $M$ is properly embedded in $\rth$. Our goal in this
section is to prove that after a rotation and a replacement of $M$
by a subend, $M$ is conformally diffeomorphic to $D(\infty ,1)$, its
height differential extends meromorphically across the puncture
$\infty $ with a double pole, and the meromorphic Gauss map $g\colon
D(\infty ,1) \to \C \cup \{\infty \}$ of $M$ has the form $g(z)= z^k
e^{H(z)}$, where $k\in \Z $ and $H$ is holomorphic in $D(\infty
,1)$.

In Theorem 1.2 of~\cite{mpr10}, Meeks, P\'erez and Ros proved that a
complete embedded minimal surface with compact boundary fails to
have quadratic decay of curvature\footnote{A properly embedded
surface $M\subset \rth$ has {\em quadratic decay of curvature} if
the function $p\in M\mapsto |K_M|(p)\, |p|^2$ is bounded, where
$|K_M|$ denotes the absolute Gaussian curvature function
 of $M$.} if and only if it does not have finite total curvature.
Since we are assuming that $M$ has infinite total curvature, then it
does not have quadratic decay of curvature.

Let $\{ \l_n\} _n \subset \R^+$ be a sequence with $\l _n\to 0$
as $n\to \infty $. We next analyze the structure of the limit of
(a subsequence of) the $\l _nM$. In our setting with compact
boundary we need to be careful at this point, since compactness
results as in Colding and Minicozzi~\cite{cm23,cm25} cannot be
applied directly. In the sequel, $\B (q,R)$ will denote the open ball of
radius $R>0$ centered at $q\in \R^3$, and $\overline{\B}(q,R)$ its closure.
\begin{lemma}
\label{lema3.1}
The sequence $\{ \l_nM\} _n$ has locally positive injectivity radius in $\R^3-\{
\vec{0}\} $, in the sense that for every $q \in \R^3-\{ \vec{0}\} $,
there exists $\ve _q\in (0,|q|)$ and $n_q\in \N $ such that for $n>n_q$, the
injectivity radius functions of the surfaces  $\l _nM$ restricted to $\B (q,\ve _q)
\cap (\l _nM)$ form a sequence of functions
which is uniformly bounded away from zero.
\end{lemma}
\begin{proof}
Arguing by contradiction, suppose that there exists a point $q\in \R^3-\{ \vec{0}\} $
and points $q_n\in \l _nM$ converting to $q$ as $n\to \infty $, such that
the injectivity radius functions of $\l _nM$ at $q_n$ converges to zero as $n\to \infty $.
Since the $\l _nM$ have nonpositive Gaussian curvature, then
(after extracting a subsequence)
we can find a sequence of pairwise disjoint embedded geodesic loops
$\g '_n\subset \B (q,1/n)\cap (\l _nM)$ with at most one
cusp at $q_n$, and with the lengths of these loops tending to zero as $n$ tends to infinity.
These curves $\g '_n$ produce related pairwise disjoint
embedded geodesic loops $\g _n$ on $M$ with at most one cusp, and
the sequence $\{ \g _n\} _n$ diverges in $M$. By the Gauss-Bonnet
formula, the curves $\g _n$ are homotopically nontrivial and hence
topologically parallel to the boundary of~$M$ (since $M$ is an annulus). Therefore, the Gauss-Bonnet
formula implies that $M$ has finite total curvature, which is a
contradiction. This contradiction proves the lemma.
\end{proof}

Since the surface $M$ does not have quadratic decay of curvature,
there exists a divergent sequence of points $\{ p_n\} _n\subset M$
(hence $p_n$ also diverges in $\R ^3$)
such that $|K_M|(p_n)\, |p_n|^2\to \infty $ as $n\to \infty $.
Consider the sequence of positive numbers $\l _n=|p_n|^{-1}\to 0$.
By Lemma~\ref{lema3.1}, the sequence $\{ \l_nM\} _n$
has locally positive injectivity radius in $\R^3-\{ \vec{0}\} $. This
property lets us apply Theorem~1.6 of~\cite{mpr11}
to the sequence of surfaces with boundary
$M_n=(\l _nM)\cap \overline{\B }(\vec{0},n) \subset \R^3-\{ \vec{0}\} $
and to the closed countable set $W=\{ \vec{0}\} $; since the
$M_n$ have genus zero and compact boundary,
then item~7.3 of Theorem~1.6 of~\cite{mpr11} ensures that
after extracting a subsequence, the
$M_n$ converge as $n\to \infty $ to a foliation ${\cal
F}$ of $\rth$ by parallel planes. The convergence of the $M_n$ to ${\cal F}$
is $C^{\a}$, $\a \in (0,1)$, away
from a set $S$ that consists of one or two straight lines orthogonal to the planes in ${\cal
F}$. The case of two lines in $S$ does not occur by the following reason;
if it occurs, then there is a limit parking garage structure with two oppositely
oriented columns (see Section~3 of~\cite{mpr14} for the notion of parking garage structure),
and the surfaces $M_n$ contain closed geodesics $\g _n$ that converge with multiplicity two to
a straight line segment that joins orthogonally the two column lines (see item~(5.2) in
Theorem~1.1 of~\cite{mpr14} for a similar statement, whose proof can be adapted to this
situation). After passing to a subsequence, the geodesics $\G _n=|p_n| \g _n\subset M$
are pairwise disjoint, which is impossible by the Gauss-Bonnet formula applied to the
annulus $A(n)\subset M$ bounded by $\G_1\cup \G_n$, as $M$ has infinite total curvature.
Therefore, $S$ consists of a single straight line.

Also note that under these homothetic shrinkings of $M$, the boundary
components  $\partial _n$ of $M_n$ corresponding to the homothetically shrunk
boundary component of $M$ collapse
into the origin $\vec{0}$ as $n\to \infty$. We claim that the
line $S$ passes through $\vec{0}$. To prove this,
suppose the claim fails and take $R>0$ such that $S$ does not intersect the closed
ball $\overline{\B }(\vec{0},4R)$.  Assume that $n$
is chosen sufficiently large so that1
$\partial _n \subset \B(\vec{0},R)$. Recall that we defined in (\ref{eq:cyl}) the vertical
solid cylinder $C(R)$ of radius $R>0$ about the $x_3$-axis.
Since the surfaces $M_n$ are converging away from $S$
to a minimal parking garage structure
of $\rth$ with associated foliation by horizontal planes, and $S$ does not intersect $\overline{\B}(\vec{0},4R)$,
then for $n$ large, the component $\Delta_n$ of $ M_n\cap C(R,R)$ that
contains $\partial_n$ (with the notation in (\ref{eq:cylcap})) must be a smooth compact
surface with its entire boundary in   $\partial C(R)$; furthermore,
every component of $\partial \Delta_n$ is a graph over
the circle  $\partial C(R)\cap \{x_3=0\}$ and has total curvature less than
 $3\pi$. As each $M_n$ is contained in a minimal
annulus, then the convex hull property implies that $\Delta_n$ is an
annulus. Since the total geodesic curvatures
of $\partial \Delta_n$ are uniformly bounded, the Gauss-Bonnet formula implies that
the total Gaussian curvatures of the compact annuli $\Delta_n$ are also uniformly bounded.
However, as the surface $M$ has infinite total curvature,
then the total Gaussian curvatures of the compact annuli $\Delta_n$
are not uniformly bounded. This contradiction proves that
the line $S$ passes through~$\vec{0}$.

Furthermore, this limit foliation ${\cal F}$ can be seen to
be independent of the sequence of positive numbers $\l_n$
(this comes from the fact that when one patches together
multivalued minimal graphs coming from different
scales, then the obtained surface is still a multivalued
graph over a fixed plane, see Part II in~\cite{cm21}
for more details). Hence, after a rotation of the initial
surface $M$ in $\R^3$, we may assume that the planes in ${\cal F}$
are horizontal. These arguments of patching together
multivalued minimal graphs coming from different scales allows us to
state the following result.

\begin{lemma}
\label{lema3.2}
There exist positive numbers $\de ,R_1$ and a solid vertical hyperboloid,
\begin{equation}
\label{eq:hyp}
{\cal H}={\cal H}(\de ,R_1)=\{ (x_1,x_2,x_3)\in \R^3\ | \ x_1^2+x_2^2<\de ^2x_3^2+R_1^2\}
\end{equation}
such that  $M-{\cal H}$ consists of
two multivalued graphs (the projection $(x_1,x_2,x_3)\mapsto (x_1,x_2)$ restricts
to a local diffeomorphism onto its image inside the $(x_1,x_2)$-plane $P$).
\end{lemma}

\begin{remark}
The arguments above use the uniqueness of the helicoid (when
applying Theorem~1.6 in~\cite{mpr11}). However, for a
simply connected surface $\h{M}$ without boundary (not an annulus
$M$ with compact boundary), both the convergence
of $\l _n\widehat{M}$ to a foliation of $\R ^3$
 by planes and the validity of Lemma~\ref{lema3.2} (after a rotation in $\R^3$) follow
directly from the main results in~\cite{cm23}, without reference to the discussion
in the previous paragraphs.
We make this remark here to warn the reader that we will be applying
the next proposition, which depends on Lemma~\ref{lema3.2}, in the simplified
proof of Theorem~\ref{t4.1} below on the uniqueness of the helicoid.
\end{remark}

By work of Colding and Minicozzi, any embedded minimal $k$-valued graph
with $k$ sufficiently large contains a $2$-valued subgraph which can be
approximated by the multivalued graph of a helicoid with an additional
logarithmic term. This approximation result, which appears as
Corollary~14.3 in Colding and Minicozzi~\cite{cm34}, together with
estimates for the decay of the separation between consecutive sheets
of a multivalued graph (equations (18.1) in~\cite{cm34} or (5.9)
in~\cite{cm22}, which in turn follow from Proposition~5.5
in~\cite{cm22}, and from Proposition II.2.12 and Corollary~II.3.7
in~\cite{cm21}), are sufficient to prove that, with the notation in Lemma~\ref{lema3.2},
each of the two
multivalued graphs $G_1$, $G_2$ in $M -{\cal H}$ contains
multivalued subgraphs $G_1'$, $G_2'$, respectively given by functions
$u_1(r,\t)$, $u_2(r,\t)$, such that $\frac{\partial
u_i}{\partial \t}(r, \t)>0$ for every $(r,\theta )$ and $i=1,2$
(resp. $\frac{\partial
u_i}{\partial \t}(r, \t)<0$ when the multivalued graphs
$G_1'$, $G_2'$ spiral downward in a counter clockwise manner); furthermore
the boundary curves of the subgraphs $G_1'$, $G_2'$ can be assumed
to lie on the boundary of another hyperboloid ${\cal H}'={\cal H}(\de ,R_2)$ for some
$R_2>R_1$. This
observation concerning $\frac{\partial u_i}{\partial \t}(r,
\t)>0$ was also made by Bernstein and Breiner using the same
arguments, see Proposition~3.3 in~\cite{bb1}. We refer the reader to
their paper for details.

With the above discussion in mind, we are ready to describe the main
result of this section.

\begin{proposition}
\label{p3.1}
After a rotation of $M$ so that the foliation
${\cal F}$ is horizontal, then:
\begin{enumerate}[1.]
\item There is a subend $M'$ of $M$ such that each horizontal plane
$P_t =x_3^{-1}(t)$ in ${\cal F}$ intersects $M'$ transversely in
either a proper curve at height $t$ or in two proper arcs,
each with its single end point on the boundary of $M'$.  More specifically,
there exists a compact solid cylinder $C(R,h)$ 
with $\partial M \subset {\rm Int}(C(R,h))$ whose top and bottom
disks each intersects $M$ transversely in a compact arc and whose
side $\{ x_1^2+x_2^2=R^2, |x_3|\leq h\} $ intersects $M$
transversely in two compact spiraling arcs $\b_1(\t)$, $\b_2(\t)$
with $\langle \frac{d\b_i}{d\t },(0,0,1)\rangle >0$, $i=1,2$ (after a possible replacement
of $M$ by $-M$), where $\t$ is the multivalued angle parameter over
the circle $\partial C(R,h) \cap P_0$,  and so that $M'= M-{\rm
Int}(C(R,h))$ is the desired subend.

\item The end $M'$ is conformally diffeomorphic to
$D (\infty ,1)$.

\item There exists $k \in \Z$ and a holomorphic function
$H$ on $D(\infty ,1 )$ such that the stereographically projected
Gauss map of $M'$ is $g(z)= z^k e^{H(z)}$ on $D(\infty ,1)$.
Furthermore, the height differential $dh=dx_3+idx_3^*$ of $M'$
extends meromorphically to $D(\infty ,1)\cup \{ \infty \} $ with a
double pole at $\infty $.
\end{enumerate}
\end{proposition}

\begin{proof}
We first construct the cylinder $C(R,h)$. Choose $R$ and $h$ large
enough so that $\partial M \subset \Int(C(R,h))$ and so that the
side $C(R,h) \cap \partial C(R)$
intersects the multivalued graphs $G_1'$, $G_2'$, transversely in two
compact spiraling arcs $\b_1(\t)$, $\b_2(\t)$ with $\langle \frac{d\b_i}{d\t },(0,0,1)\rangle >0$, 
$i=1,2$ (we may assume this sign of $\frac{d\b_i}{d\t
}$ after replacing $M$ by $-M$).
Note that by increasing the radius $R$ and height $h$, the sets
$M\cap C(R,h)$ define a compact exhaustion of $M$, since $M$ is
proper. These $R,h$ are taken so that $R\gg h$ in order to ensure
that  $C(R,h) \cap \partial C(R)$ is  disjoint from the hyperboloid
${\cal H}$ defined by (\ref{eq:hyp}), and $C(R,h)\cap \partial C(R)$
intersects the multivalued graphs $G_1'$, and $G_2'$ in
the above spiraling-type arcs $\b_1,\b _2$. Also take $R,h$ large enough so that the
component of $M\cap C(R,h)$ which contains $\partial M$, also
contains $\b _1,\b _2$ in its boundary. Without loss of generality,
we can assume that the top and bottom disks $D_h\subset P_h$,
$D_{-h}\subset P_{-h}$ of $C(R,h)$
each intersects $M$ transversely. The convex hull property implies that
there are no closed curve components in $(D_h\cup D_{-h})\cap M$, and thus
$(D_h\cup D_{-h})\cap M$ consists of
two compact arcs $\a_T\subset D_h$, $\a_B\subset D_{-h}$,
see Figure~\ref{figure1}.
\begin{figure}
\begin{center}
\includegraphics[height=6.4cm]{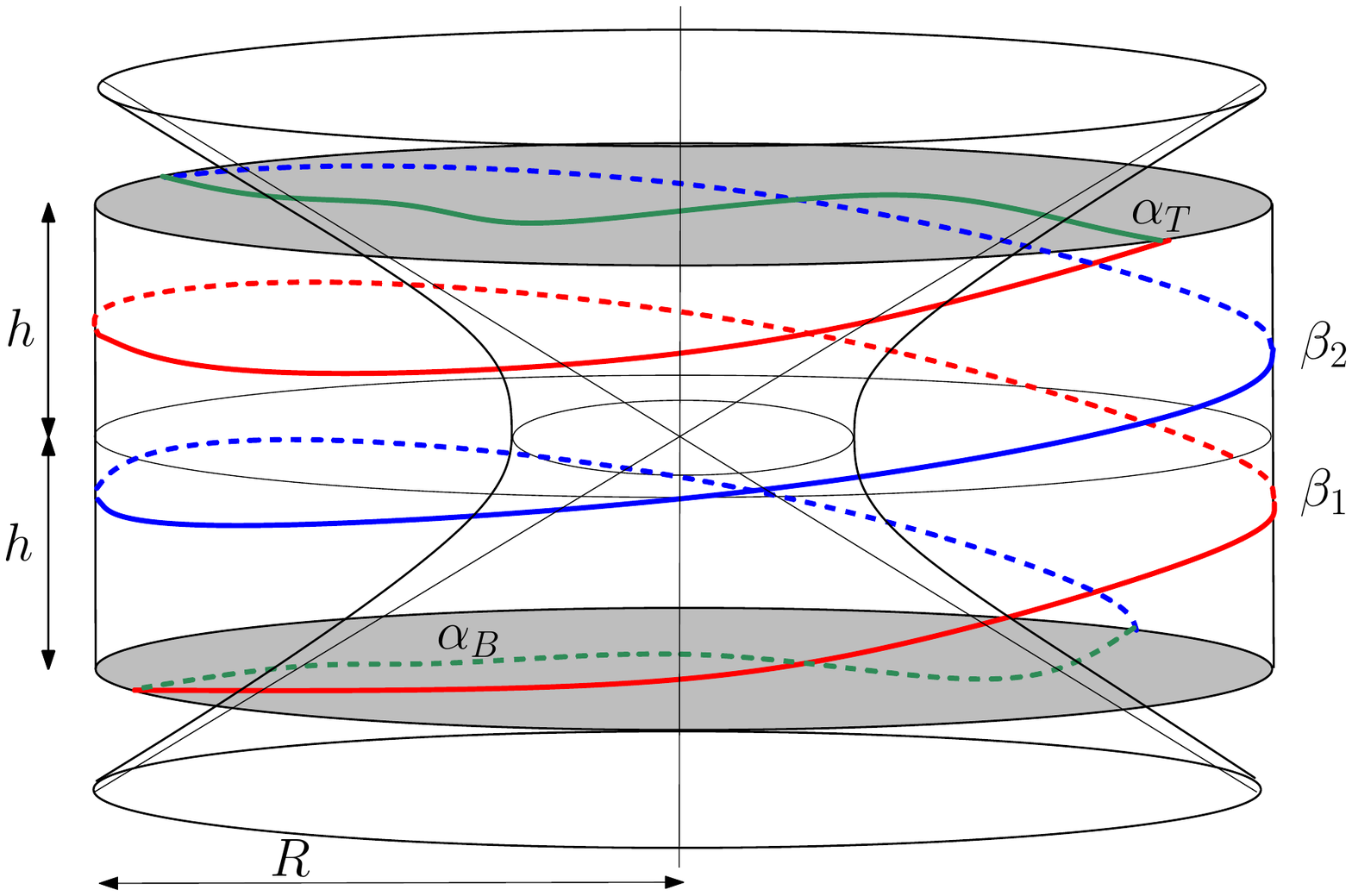}
\caption{$x_3$ is monotonic along the curves $\b_1,\b _2=
M'\cap [
\partial C(R)\cap \{ 0<x_3<h\} ]$ (red, blue), and constant along
the top and bottom arcs $\a_T,\a_B=M'\cap [\partial C(R,h)- \{ 0<x_3<h\} ]$
(green).
} \label{figure1}
\end{center}
\end{figure}
Hence, $M' =  M-\Int(C(R,h))$ is a annular subend of $M$.
Furthermore, each of the planes $P_{-h}$, $P_h$ intersects $M'$ in a
connected proper curve $\widetilde{\a}_B, \widetilde{\a}_T$
that contains $\a_B, \a_T$, respectively;
this is because each of $\widetilde{\a}_B, \widetilde{\a}_T$ intersects each of the
cylinders $\partial C(R')$, $R' \geq R$, at exactly two points.  By our
choices of $G_1'$, and $G_2'$, each plane $P_t$, $t \in (-h,h)$,
intersects $M'$ in a pair of proper arcs each with one end point in
$\partial M'$.  Since each of the planes $P_t $, $|t|\geq h$, intersects
each of $G_1'$, and $G_2'$ transversely in a single
proper arc with end point in the boundary of the respective
multivalued graph, then elementary Morse theory and the convex hull property
imply that each of these planes intersects $M'$ transversely in a
single proper curve. Since $M'$ intersects some horizontal plane in
a connected proper curve, then the remaining statements in the proposition follow directly from
Corollary~1.2 in~\cite{mpe5}.
\end{proof}

\section{Uniqueness of the helicoid.}
In this section we will apply Proposition~\ref{p3.1} to simplify the
original proof by Meeks and Rosenberg of the uniqueness of the
helicoid among properly embedded, simply connected, nonflat minimal
surfaces in $\R^3$. The proof we give below only uses the one-sided
curvature estimates by Colding and Minicozzi (Theorem~0.2
in~\cite{cm23}), Proposition~\ref{p3.1} in this paper, and some
arguments taken from the end of~\cite{mr8}. Since the uniqueness of
the helicoid given in~\cite{mr8} is not used in the proof by Colding
and Minicozzi that completeness implies properness for embedded,
finite topology minimal surfaces, then we can replace the hypothesis
of properness by the weaker one of completeness in the above
statement of uniqueness of the helicoid.
\begin{theorem} \label{t4.1}
If $M\subset \R^3$ is a complete, embedded, simply connected
minimal surface, then
$M$ is either a plane or a helicoid.
\end{theorem}
\begin{proof}
Assume $M$ is not flat. By Proposition~\ref{p3.1}, the conformal
type of $M$ is $\C $ and its height differential $dh=dx_3+idx_3^*$
extends meromorphically to $\infty $ with a double pole. Therefore
$dh=\l \, dz$ for some $\l \in \C^*$, which means that
$h=x_3+ix_3^*$ gives a global conformal parametrization of $M$. In
the sequel, we will use this parametrization from $\C $ into $M$, so
that $h(z)=z$. Since $dh=dz$ has no zeros, we conclude that the
Gauss map $g$ misses $0,\infty $ on the entire surface $M$. Since $M$
is simply connected, then $g$ lifts through the natural exponential
map $e^w\colon \C \to \C ^*$ and thus, $g(z)=e^{H(z)}$ for some
entire function $H$.

The next arguments show that $H$ is a linear function of $z$, which
clearly implies that $M$ is a helicoid (recall that the only
associate surface\footnote{An associate surface to $M$ with
Weierstrass data $(g,dh)$ is any of the (in general, multivalued)
immersions with Weierstrass data $(g,e^{i\t }dh)$, where $\t \in
[0,2\pi )$.} to the helicoid which is embedded as a mapping is the
helicoid itself). Suppose $H$ is not a linear function and we will
obtain a contradiction. Applying Sard's theorem to the third
component $N_3$ of the spherical Gauss map of $M$, we deduce that
there exists a latitude $\g \subset \esf^2$ arbitrarily close to the
horizontal equator, such that $g^{-1}(\g )$ does not contain any
branch point of $N_3$. Therefore, $g^{-1}(\g )$ consists of a proper
collection of smooth curves. Note that there are no compact
components in $g^{-1}(\g )$ since $M$ is simply connected and $g$ misses
$0,\infty $. Also note that $H|_{g^{-1}(\g )}$ takes values along a certain
line $l\subset \C $ parallel to the imaginary axis, and the
restriction of $H$ to every component of $g^{-1}(\g )$ parameterizes
monotonically an arc in the line $l$. Hence, Picard's theorem
implies that either $H$ is a polynomial of degree $m\geq 2$ and $g^{-1}(\g )$
has $m$  components, or $g^{-1}(\g )$ has an infinite number of
components (equivalently, $H$ is entire with an essential
singularity at $\infty $).

The remainder of the proof breaks up into Cases A, B below. Each of
these cases uses  results of Colding and Minicozzi from~\cite{cm23}
on the geometry of $M$. By the one-sided curvature estimates of
Colding and Minicozzi~\cite{cm23}, there exists a solid vertical
hyperboloid of revolution ${\cal H}={\cal H}(\de ,R)\subset \R^3$
(defined by (\ref{eq:hyp})) such that $M-{\cal
H}$ consists of two infinite valued graphs whose norms of their gradients are
less than one. In particular, the curves in $g^{-1}(\g )$, when considered
to be in the surface $M\subset \R^3$, are contained in ${\cal H}$.
In Case B below, where  $H$ is a polynomial of degree  $m\geq 2$ and
$g^{-1}(\g )$ has $m$  components, we may assume that at least one of
the two ends of ${\cal H}$, say the top end, contains at least two
of the ends of curves in $g^{-1}(\g )$. In Case~A, where $g^{-1}(\g )$ is an infinite
proper collection of curves inside ${\cal H}$, then at least one of
the two ends of ${\cal H}$, say the top end, must contain an
infinite number of ends of curves in $g^{-1}(\g )$. Thus in Case~A for every
$m\in \N$, there exists a large $T=T(m)>0$ such that every plane $P_t=\{
x_3=t\} $ with $t>T(m)$ intersects at least $m$ of the curves in $g^{-1}(\g )$
.

\vspace{.3cm}
\noindent {\bf Case A:}  $g^{-1}(\g )$
has an  infinite number of components.
\par
\noindent
Using the rescaling argument in the proof of
Proposition~5.2 (page 754) of~\cite{mr8}, we can produce a blow-up
limit $\widetilde{M} = \lim_k \lambda_k (M-p_k)$ of $M$ that is a
properly embedded, simply connected minimal surface with infinite
total curvature and bounded Gaussian curvature; with some more detail,
the points $p_k\in M$ are points of nonzero Gaussian
curvature, the limit surface
$\widetilde{M}$ passes through the origin with this point as a maximum
of its absolute Gaussian curvature, $\lim_{k\to \infty} x_3(p_k)=\infty$
and the intersection of $\widetilde{M}$ with
$\{ x_3=0\} $ corresponds to the limit of scalings of the intersections of $M$ with
the sequence of planes $P_{t_k}$ with $t_k=x_3(p_k)\to \infty $ and
 $P_{t_k}$ intersects at least $k$ components of ${g}^{-1}(\g )$.
The new
surface $\widetilde{M}$ has the same appearance as that of $M$ with
two infinite  valued graphs. In particular, if
$\widetilde{g}$ is the Gauss map of $\widetilde{M}$, then
$\widetilde{g}^{-1}(\g )$ consists of a proper collection of curves which
is contained in the related solid hyperboloid $\widetilde{\cal H}$
for $\widetilde{M}$. In particular, there are only a finite number
of components of $\widetilde{g}^{-1}(\g )$ that intersect $\{
x_3=0\} $,  and  we may also assume that
$\g$ is chosen so that $\widetilde{g}^{-1}(\g )$
intersects  $\{x_3=0\} $ transversely in a finite number $n_0$ of points.
By the extension property in~\cite{cm21}, the related
multivalued graphs in the domains $p_k+\frac{1}{\lambda_k}\left[
\widetilde{M}\cap (\R^3-\widetilde{\cal H})\right] $ near $p_k$
extend with the norms of gradients at most 2 all the way to infinity in $\rth$ (we used a
similar argument in the sentence just before the statement of Lemma~\ref{lema3.2}).
Therefore, each of the infinitely many components of $g^{-1}(\g)$ must
intersect transversely the plane $\{ x_3=x_3(p_k)\} $ in the compact disk
$p_k+\frac{1}{\lambda_k}\left[ \{ x_3=0\} \cap \widetilde{\cal
H}\right] $ for $k$ sufficiently large. In particular,
since $\widetilde{g}^{-1}(\g )$ intersects $\{ x_3=0\} $ transversely in $n_0$ points,
we find that $g^{-1}(\g)$ must have at most $n_0$ components,
which is a contradiction. This contradiction implies that Case A
does not occur.

\vspace{.3cm}
\noindent {\bf Case B:} $H$ is a
polynomial of degree  $m\geq 2$ and  $g^{-1}(\g)$ has $m$  components.
\par
\noindent Again using the rescaling argument in the proof of
Proposition~5.2 of~\cite{mr8}, we  produce a simply connected
minimal surface $\widetilde{M} = \lim_k \lambda_k (M-p_k)$ with
infinite total curvature and absolute Gaussian curvature at most
1 and equal to 1 at $\vec{0}\in\wt{M}$, and with the
property that the intersection of $\widetilde{M}$ with $\{ x_3=0\} $
corresponds to the limit of intersections of $M$ with a sequence of
planes $P_{t_k}$ with $t_k=x_3(p_k)\to \infty $. As in Case~A above,
$\widetilde{M}$ has the same appearance as that of $M$ with two infinitely
valued graphs, and if $\widetilde{g}$ denotes
the Gauss map of $\widetilde{M}$, then $\widetilde{g}^{-1}(\g )$
consists of a finite number  of curves, each of which is contained
in the related solid hyperboloid $\widetilde{\cal H}$ for
$\widetilde{M}$. The same arguments as in the beginning of this
proof of Theorem~\ref{t4.1} demonstrate that the conformal type of
$\widetilde{M}$ is $\C $ and $\widetilde{M}$ can be conformally
parameterized in $\C $ with height differential $dz$ and Gauss map
$\widetilde{g}(z)=e^{\widetilde{H}(z)}$, for some entire function
$\widetilde{H}$.

We now show how the fact that $\widetilde{M}$ has bounded Gaussian
curvature gives that $\widetilde{H}$ is linear (and so,
$\widetilde{M}$ is a helicoid): Since the arguments in Case A above
can be applied to $\widetilde{M}$, we conclude that $\widetilde{H}$
cannot have an essential singularity at $z=\infty $ and thus,
$\widetilde{H}$ is a polynomial. Suppose that $\widetilde{H}$ has
degree strictly greater than one, and we will get a contradiction.
As $\widetilde{g}(z)$ is an entire function with an essential
singularity at $\infty $ and $\widetilde{g}$ misses $0,\infty $ (by
the open mapping property since $g$ also omits $0,\infty $), then
$\widetilde{g}$ takes any value in $\C -\{ 0\}$ infinitely often in
any punctured neighborhood of $z=\infty $. Thus, there exists a
sequence $z_k\in \C $ diverging to $\infty $ as $k\to \infty $, such
that $\widetilde{g}(z_k)=1$ for all $k$. Recall that the Gaussian
curvature $\widetilde{K}$ of $\widetilde{M}$ in terms of its
Weierstrass data $\left( \widetilde{g}(z)=e^{\widetilde{H}(z)}, dz
\right) $ is given by
\[
\widetilde{K}=-16|\widetilde{H}'|^2\frac{|\widetilde{g}|^4}{(1+|\wt{g}|^2)^4};
\]
hence $\widetilde{K}(z_k)=-|\widetilde{H}'(z_k)|^2$ tends to $-\infty $ as $k\to \infty $,
which is a contradiction. This contradiction proves that $\widetilde{H}$
is linear and so, $\widetilde{M}$ is a vertical helicoid.

Finally, the extension property in~\cite{cm21} implies that
the related horizontal multivalued graphs in the domains
$p_k+\frac{1}{\lambda_k}\left[ \widetilde{M}\cap (\R^3-\widetilde{\cal H})\right] $
near $p_k$ extend all the way to infinity as multivalued graphs with the norm of the gradient less than 2.
Reasoning as in the
proof of Case A, for $k$ large,
 ${g}^{-1}(\g )$  intersects transversely the plane $\{ x_3=x_3(p_k)\} $ in a subset of the compact disk
$p_k+\frac{1}{\lambda_k}\left[ \{ x_3=0\} \cap \widetilde{\cal
H}\right] $ that consists of a single point near $p_k$,
since on the helicoid $\wt{\cal{H}}$, \,$\wt{g}^{-1}(\g)$ is a helix
or a vertical line that intersects the
plane $\{x_3=0\}$ transversely in a single point.  But by  our choices,
the top end of $\cal{H}$ contains at least two ends of ${g}^{-1}(\g)$, each of which must intersect
transversely the plane $\{ x_3=x_3(p_k)\} $ for $k$ large.   This is a contradiction
 which completes the proof of Case~B, and so, it also finishes the proof of the theorem.
\end{proof}

The next result is an immediate consequence of Theorem~\ref{t4.1} and of
the discussion before the statement of Proposition~\ref{p3.1}; also see the
related rescaling argument in the proof of Proposition~5.2 of~\cite{mr8}.

\begin{corollary} \label{c.hel}
Let $M\subset \rth$ be a complete, embedded minimal annulus and
$X\colon D(R,\infty)\to \rth$ be a conformal parametrization of $M$
with Weierstrass data
\[
\left(
g(z)=z^ke^{H(z)},\ dh=dx_3+idx_3^*\right) ,
\]
where $H(z)$ is
holomorphic, $dh$ has no zeroes in $D(R,\infty )$ and $dh$ extends
holomorphically across $\infty$ with a double pole at $\infty$. Then, there exists
positive numbers $\de ,R$ and sequences $\{p_n^{+}\}_{n\in \N}$,
$\{p_{n}^{-}\}_{n\in \N}$ in $M\cap {\cal H}(\de ,R)$,
$\{\lambda_n^{+}\}_{n\in \N}$, $\{ \l_n^{-}\}_{n\in \N}\subset (0,\infty )$  such that:
\begin{enumerate}[1.]
 \item  $g(p_n^+)=g(p_n^-)=1\in \C$ \, for all $n\in \N$, and $ x_3(p_n^+)\to \infty$, \ $x_3(p^-_n)=-\infty$ as $n\to \infty $.
\item The surfaces $\l_n(M-p_n^{+})$ converge smoothly on compact
subsets of $\rth$ as $n\to \infty $ to a vertical helicoid ${\bf
H}\subset \rth$ which contains the $x_3$ and the $x_2$-axes.
\end{enumerate}
\end{corollary}

\section{The Weierstrass representation
of embedded minimal surfaces of finite topology.} In this section we
describe an improvement of the analytic description of a complete
injective minimal immersion of $\esf^1 \times[0,\infty)$ into $\rth$
that was given in Proposition~\ref{p3.1}.  This improvement is
given in the next theorem, which implies that $M$ has finite type.
Theorem~4 in~\cite{hkp1} states that a complete, embedded minimal
annular end of finite type with Weierstrass data as given in the
next theorem and with zero flux, is asymptotic to the end of a
vertical helicoid. We will prove this last fact again,
independently, later with the stronger result described in
Theorem~\ref{main}. Note that if $M$ has zero flux, then the number
$\l $ in item~2 below is zero and so, the next theorem will complete
the proof of Theorem~\ref{th1.1} except for item~3c; this
item will follow from items~\ref{it8}, \ref{it9} of Theorem~\ref{main}.

\begin{theorem}
\label{ft}
Suppose $M\subset \rth$ is a complete, injectively immersed minimal annulus
with compact boundary and infinite total curvature.
After a homothety and rotation, $M$ contains a subend $M'$ with the following
Weierstrass data:
\[
\left( g(z)=e^{iz+f(z)},\ dh=\left( 1+\frac{\l }{z-\mu }\right) dz\right) , \quad z\in D(\infty ,R),
\]
where $R>0$, $f\colon D(\infty ,R)\cup \{ \infty \} \to \C $ is
holomorphic with $f(\infty)=0$, $\l \in [0,\infty )$ and $\mu \in \C$.
\end{theorem}
\begin{proof}
After a rotation of $M$ and a replacement by a
subend, we may assume by Proposition~\ref{p3.1} that
$M$ is the image of a conformal minimal embedding $X\colon D(\infty,1)\to \rth$,
the meromorphic Gauss map of $X$ is equal to $g(z)=z^ke^{H(z)}$ where
$k\in \Z$ and $H\colon D(\infty,1)\to \C$ is some holomorphic function, and
the holomorphic height differential of $X$ is the one-form $dh=dx_3+idx_3^*$
which has no zeroes and which extends meromorphically to $D(\infty ,1)\cup\{\infty\}$
with a double pole at infinity.

\begin{assertion} In the situation above, $k=0$. 
\end{assertion}
\begin{proof} Since $g(z)=z^ke^{H(z)}$, then $k$ is  the
winding number of $g|_{\partial D(\infty, 1)}$ in $\C-\{0\}$.
We will show that the winding number of $g|_{\partial D(\infty,1)}$ is zero
by proving that there exists a simple closed curve $\G\subset D(\infty,1)$ that is the
boundary of an end representative of $D(\infty, 1)$ and such that
$g|_{\G}$ has winding number $0$; since such a simple
closed curve $\G$ is homotopic to $\partial D(\infty, 1)$, it will follow
that the winding number of $g|_{\partial D(\infty, 1)}$  also vanishes.

By Proposition~\ref{p3.1}, there exist $R',h'>0$ such that $M'=M-\mbox{Int}(C(R',h'))$
is a subend of $M$, each horizontal plane $P_t=\{ x_3=t\} $ intersects $M'$
transversely in either a proper curve (if $|t|\geq h'$) or in two proper
arcs, each with its single end point on the boundary of $M'$ (if $|t|<h'$).
By Corollary~\ref{c.hel}, there exist $\de, R_1>0$
and a point $p^-\in M\cap {\cal H}(\de ,R_1)$
with $x_3(p^-)\ll -1$ such that $g(p^-)=1$ and
$M$ has the appearance of a vertical
helicoid around $p^-$. By the extension property in~\cite{cm21}, the
two resulting multivalued graphs in $M$ around $p^-$ extend indefinitely
sideways  as very flat multivalued graphs. Consider the intersection of $M$
with the tangent plane $T_{p^-}M$. This intersection is an analytic
1-dimensional complex which contains an almost horizontal proper smooth
 curve $\a_B$ passing through $p^-$, which we may assume is globally parameterized
 by its $x_2$-coordinate. In this case, $\a_B$ near $p^-$ plays the role of the $x_2$-axis
contained  in the helicoid ${\bf H}$ that appears in
Corollary~\ref{c.hel}. It follows that close to $p^-$, the argument
of $g(z)$ on $\a_B$ can be chosen to be close to zero, see
Figure~\ref{figp-}.
\begin{figure}
\begin{center}
\includegraphics[height=5.5cm]{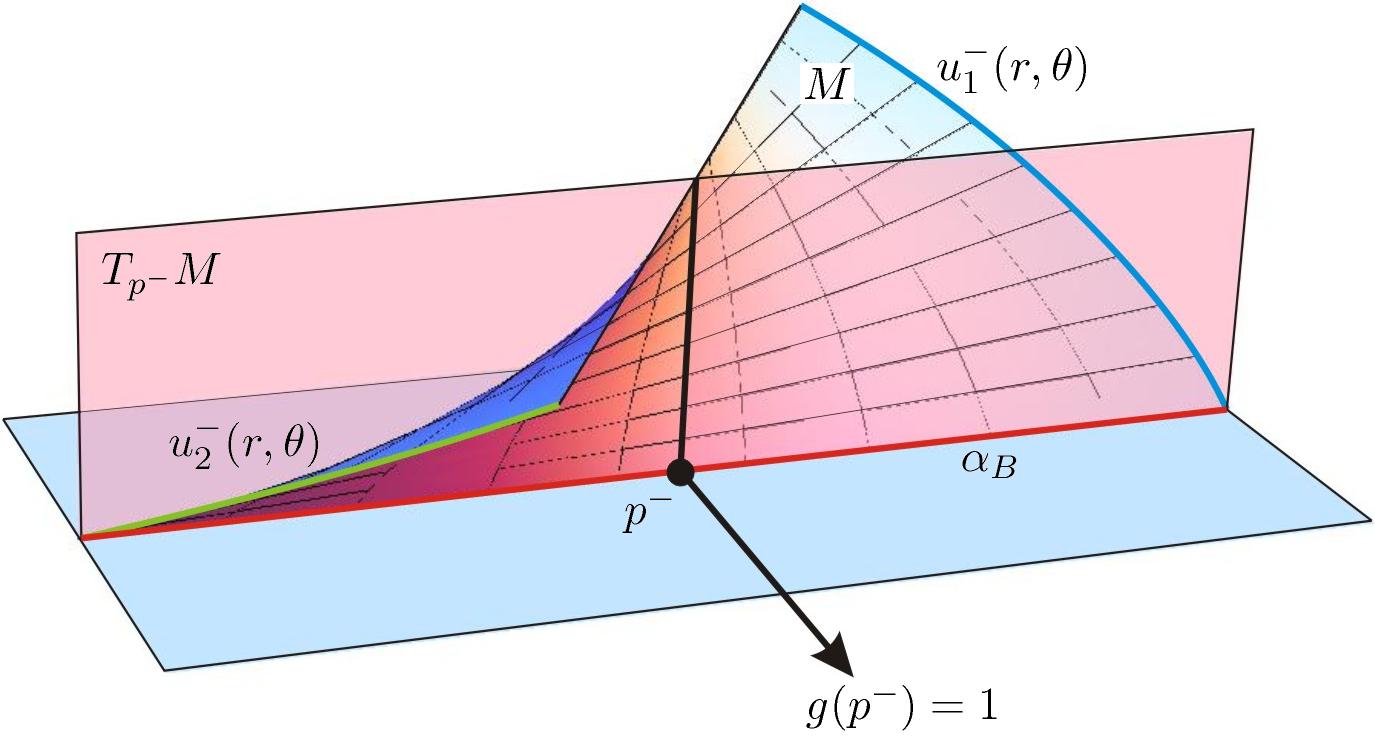}
\caption{Around $p^-$, $M$ has the appearance of a translated and
homothetically shrunk vertical helicoid, made of two multivalued graphs
$u_i^-(r,\theta )$, $i=1,2$.}
\label{figp-}
\end{center}
\end{figure}

Since the two horizontal multivalued graphs  around $p^-$ extend
horizontally all to way to infinity in $\R^3$, our earlier
discussion before the statement of Proposition~\ref{p3.1} implies
that  the related multivalued graphing functions $u_1^-,u_2^-$ with polar
coordinates centered at $p^-$, satisfy $\frac{\partial
u_i^-}{\partial \theta}>0$ for $i=1,2$ when restricted to $\a_B -
\{p^- \}$, even at points far from the forming helicoid at $p^-$.
Thus, if we consider the argument of the complex number $g(p^-)$ to
lie in $(-\frac{\pi }{2},\frac{\pi }{2})$, we conclude that the
argument of  $g(\a_B(t))$ also lies in $(-\frac{\pi }{2},\frac{\pi
}{2})$ for all $t\in \R $.

A similar discussion shows that there is
a point $p^+\in M\cap {\cal H}(\de ,R_1)$ with $x_3(p^+)\gg 1$  and
a related almost horizontal curve $\a_T\subset M\cap T_{p^{+}}M$
passing through $p^+$ and such that the argument of $g$ along
$\a_{T}$ also takes values in the interval $(2\pi j-\frac{\pi }{2},
2\pi j+\frac{\pi }{2})$, when we consider the
argument of $g(p^+)=1$ to lie in the interval $(2\pi j-\frac{\pi
}{2}, 2\pi j+\frac{\pi }{2})$ for some $j\in \N$.

Recall from the discussion just before Proposition~\ref{p3.1} that
the hyperboloid ${\cal H}'={\cal H}(\de ,R_2)$
was chosen so that in the complement of
 ${\cal H}'$,
$M$ consists of two multivalued graphs $G_1'$, $G_2'$, respectively given
by functions $u_1(r,\theta)$, $u_2(r,\theta)$, such that
$\frac{\partial u_i}{\partial \theta}(r,\theta )>0$, $i=1,2$. Now
can choose $r'$ sufficiently large so that:
\begin{enumerate}[$\bullet$]
\item The portion of $\a_T\cup\a_B$
outside of the solid vertical cylinder $C(r')=\{ x_1^2+x_2^2\leq (r')^2\}$
consists of four connected arcs in $G_1' \cup G_2'$, and the portion of
$\a_T\cup\a_B$ inside $C(r')$ consists of two compact arcs $\a_T^{r'}$,
$\a_B^{r'}$.
\item $\partial C(r')\cap (G_1'\cup G_2')$ contains
two connected spiraling arcs $\beta_1$, $\beta_2$,
each of which connects
one end point of $\a_T^{r'}$ with one end point of $\a_B^{r'}$.
\end{enumerate}

We define the simple closed curve $\G=\a_B^{r'}\cup \beta_1 \cup \a_T^{r'}\cup \beta_2
\subset M$, see Figure~\ref{figure5}. Note that $\G $ bounds a subend of $M$,
and that $\G$ can be parameterized by
a mapping $\g\colon [0,4]\to \G$, so that
\begin{enumerate}[$\bullet$]
\item $\a_{B}^{r'}$ has end points $\g(0), \g(1)$,
\item $\g(1),\g(2)$ are the end points of $\beta_1$,
\item $\g(2),\g(3)$ are the end  points of $\a_T^{r'}$, and
\item $\beta_2$ has end points $\g(3), \g(4)=\g (0)$, see Figure~\ref{figure5} left.
\end{enumerate}
\begin{figure}
\begin{center}
\includegraphics[height=5cm]{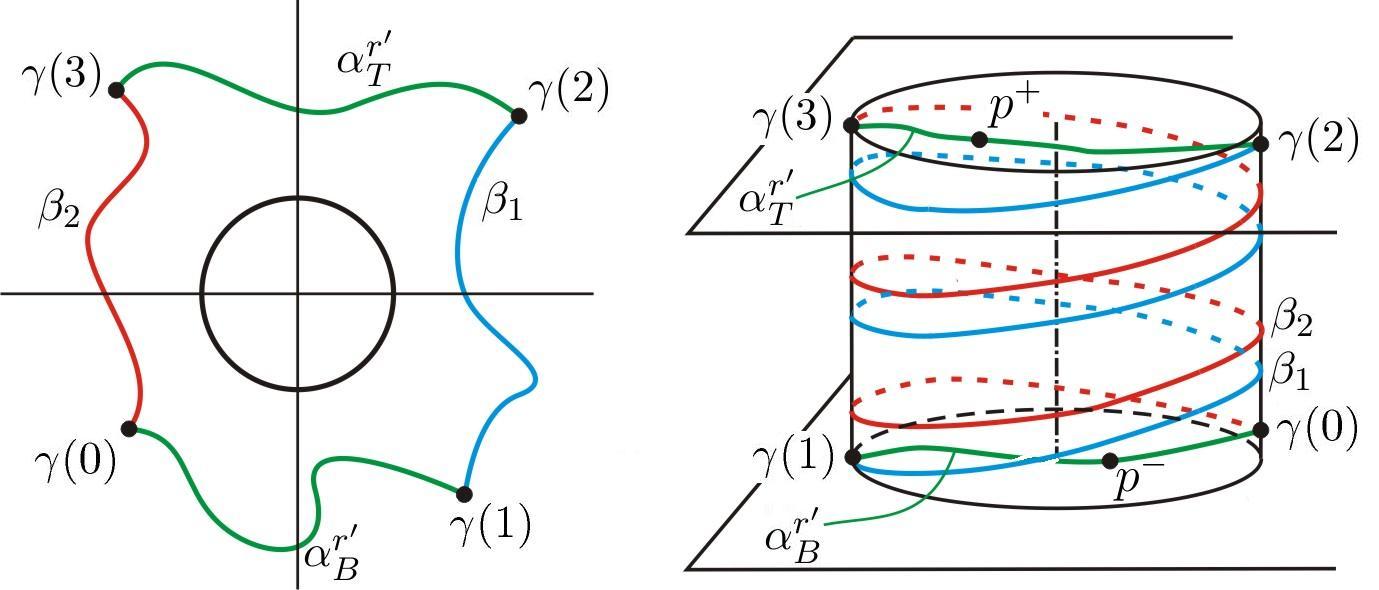}
\caption{Left: The closed curve $\G=\a_B^{r'}\cup \beta_1 \cup
\a_T^{r'}\cup \beta_2$ is homotopic to the boundary of $D(\infty
,1)$. Right: The same loop $\G $, viewed in $M\subset \R^3$;
note that $\a_B^{r'},\a_T^{r'}$ are contained in vertical
planes parallel to the $(x_2,x_3)$-plane ($\a_B^{r'}$, $\a_T^{r'}$ are not necessarily
contained in horizontal planes).}
\label{figure5}
\end{center}
\end{figure}
Notice that the condition
$\frac{\partial u_i}{\partial \theta}(r,\theta)>0$
along $\beta_1 \cup \beta_2$
implies that the argument of $g(z)$
along $\beta_1$, which we can choose to have an initial value
at $\g (1)=\beta_1 \cap \a_B^{r'}$ in the interval $(-\frac{\pi }{2},\frac{\pi }{2})$,
has its value at $\gamma (2)$ in $(2\pi n-\frac{\pi }{2}, 2\pi n+\frac{\pi }{2})$,
where $n$ is the number of $\theta$-revolutions that $\beta _1$
makes around the $x_3$-axis (when we assume that the Gauss map of $M$ is
upward pointing along $\beta_1$). Similarly, the argument of $g(z)$ along
$\beta_2$, which has an initial value in $(2\pi n-\frac{\pi }{2} , 2\pi n+\frac{\pi }{2})$, has its
ending value in $(-\frac{\pi }{2},\frac{\pi }{2})$; we are using here the earlier
observation that if the argument of $g(z)$ at one end point
of $\a^T$ or $\a^B$ lies in $(2\pi m-\frac{\pi }{2}, 2\pi m+\frac{\pi }{2})$
for some $m\in \N$, then the other end point of the same
$\a $-curve also lies in $(2\pi m-\frac{\pi }{2}, 2\pi m+\frac{\pi }{2})$. It follows
now that the winding number of $g|_{\G}$ is zero, which
completes the proof of the assertion.
\end{proof}

We continue with the proof of Theorem~\ref{ft}.
So far we have deduced that $M$ can be parameterized by  $D(\infty,1)$
with Weierstrass data $g(z)=e^{H(z)}$, $H(z)$ holomorphic on
$D(\infty, 1)$ and $dh=dx_3+idx_3^*$, where $dh$ has no zeroes and
extends to $D(\infty, 1)\cup\{\infty\}$ with a double pole at
infinity. We  claim that $H(z)=P(z)+f(z)$, where $P(z)$ is a
polynomial of degree one and $f(z)$ extends to a holomorphic
function on $D(\infty, 1)\cup\{\infty\}$ with $f(\infty)=0$. The
first case we consider, Case~A, is when $H(z)$ has an essential
singularity at $\infty$. This case cannot occur by the same arguments
as those given in the proof of Theorem~\ref{t4.1} in the similar
Case~A. The second case we consider, Case~B, is when
$H(z)=P(z)+f(z)$ where $P(z)$ is a polynomial of degree $m\geq 2$
and $f(\infty)=0$. In this case the arguments of the related Case~B
in the proof of Theorem~\ref{t4.1} show that this case cannot occur either.
This proves the claim that $H(z)=cz+d+f(z)$ where
$c,d\in \C $ and $f(\infty)=0$. Since $M$ is assumed to have infinite total
curvature, its Gauss map takes  on almost all values
of $\esf^2$ infinitely often,  which means that $c\not =0$.

Finally we will obtain the Weierstrass data that appears in the
statement of Theorem~\ref{ft}. Since $dh$ has a double pole at
infinity, then after a change of variables of the type $z\mapsto
zT(z)$ for a suitable holomorphic function $T(z)$ in a neighborhood
of $\infty $ with $T(\infty )\in \C -\{ 0\} $, we can write  $dh=(
A_0+\frac{\l }{w})\, dw$, where $A_0\in \C -\{ 0\} $ and $-2\pi \l
\in \R $ is the vertical component of the flux vector of $M$ along its
boundary. After possibly rotating $M$ in $\R^3$ by 180$^o$ about a
horizontal line, we can assume $\l \geq 0$.
After a homothety of $M$ in $\R^3$ and a suitable rotation
$w\mapsto e^{i\theta }w$ in the parameter domain, we can prescribe
$A_0$ in $\C -\{ 0\} $. As a rotation of $M$ in $\R^3$ about the
$x_3$-axis does not change $dh$ but multiplies $g$ by a unitary
complex number, then we can assume that the Weierstrass data of $M$
are $g(w)=e^{cw+d+f(w)}$ and $dh=(-ic+\frac{\l }{w})\, dw$. Finally,
the linear change of variables $cw+d=i\xi $ produces $g(\xi
)=e^{i\xi +f_1(\xi )}$, $dh=( 1+\frac{\l }{\xi +id})\, d\xi $ where
$f_1(\xi )=f(w)$. After choosing $R>0$ sufficiently large, we have
the desired Weierstrass data on $D(\infty ,R)$. This concludes the
proof of the theorem.
\end{proof}

\section{The analytic construction of the
examples $E^{a,b}$.}
\label{sec6}
 In this section we will construct certain
examples $E^{a,b}$ of complete {\it immersed} minimal annuli which
depend on parameters $a,b\in [0,\infty )$ and so that their flux vector
is $(a,0,-b)$.  The
{\it embedded} canonical examples $E_{a,b}$ referred to in
Theorem~\ref{thm1.3} will be end representatives of the corresponding examples $E^{a,b}$;
this is explained in item~\ref{it1} in Theorem~\ref{main} below.
Our construction of the examples discussed in this section will be based on the
classical Weierstrass representation.  We first review how we can deduce  the
periods and fluxes of a potential minimal annulus from its Weierstrass representation.

Consider a meromorphic function $g$ and a holomorphic one-form $dh$
on the punctured disk $D(\infty, R)$
centered at $\infty $, for some $R>0$. Suppose that
$(|g|+|g|^{-1})|dh|$ has no zeros in $D(\infty , R)$.
The period and flux along
$\partial _R=\{|z|=R\}$ of the (possibly multivalued)
unbranched minimal immersion $X\colon D(\infty, R)\to \rth$
associated to the Weierstrass data
$(g,dh)$ are given by
\[
\mbox{Per}+i \mbox{ Flux}=\frac{1}{2}\left(
\int _{\partial _R}\frac{dh}{g}-
\int _{\partial _R}g\, dh, i\int _{\partial _R}\frac{dh}{g}+i
\int _{\partial _R}g\, dh, 0\right) +
\left( 0,0,\int _{\partial _R}dh\right) \in \C^3.
\]
If we assume that the period of $X$ vanishes, i.e.
(\ref{eq:periodproblem}) holds, then the above formula
gives
\begin{equation}
\label{eq:flux}
\mbox{ Flux}=\left(
-\mbox{Im}\int _{\partial _R}g\, dh, \mbox{Re} \int _{\partial _R}g\, dh,
\mbox{Im} \int _{\partial _R}dh\right)
\equiv \left( i\int _{\partial _R}g\, dh,
\mbox{Im} \int _{\partial _R}dh\right) \in \C \times \R ,
\end{equation}
where we have identified $\R^3$ with
$\C \times \R $ by $(a_1,a_2,a_3)\equiv (a_1+ia_2,a_3)$.

It is clear that the choice $g(z)=e^{iz}$, $dh=dz$ produces the
end of a vertical helicoid.
In this section, we will consider
explicit expressions for $(g,dh)$ such that
\begin{enumerate}[1.]
\item $g(z)=e^{iz+f(z)}$, where $f\colon D(\infty ,R)\cup \{ \infty \} \to \C $ is
holomorphic and $f(\infty)=0$.
\item $dh=(1+\frac{\l }{z-\mu })\, dz$, where $\l \in [0,\infty )$ and $\mu \in \C$,
\end{enumerate}
in terms of certain parameters defining $f(z)$ and $\mu$ and we
will prove that these parameters can be adjusted so that the period
problem for $(g,dh)$ is solved, thereby defining a complete immersed
minimal annulus $E^{a,b}$; furthermore, the flux vector of the
resulting minimal immersion is $F=(a,0,-b)$ and this vector
can be chosen for every $a,b\geq 0$. In the particular case
$a=b=0$, we choose $f=\l =0$; hence $E^{0,0}$ is the end of a
vertical helicoid.
The argument for the construction of the remaining annuli $E^{a,b}$
breaks up into two cases, depending on whether the flux vector $F$
is vertical or not.
\par
\vspace{.2cm} \noindent
 {\bf Case 1: The flux is not vertical.}
\par
Consider the following particular choice of $g,dh$:
\begin{equation}
\label{eq:gdh1}
g(z)=t\, e^{iz}\frac{z-A}{z}, \qquad dh=\left( 1+
 \frac{B}{z}\right) dz, \qquad z\in D(\infty ,R),
\end{equation}
where $t>0$, $A\in \C-\{ 0\} $, $B\in [0,\infty )$ and $R>|A|$ are to be
determined.  Note that $g$ can rewritten as $g(z)=t\, e^{iz+f(z)}$
where $f(z)=\log\frac{z-A}{z}$, which is univalued and holomorphic
in $D(\infty ,R)$ since $R>|A|$. Furthermore,  $f(\infty )=0$. Also
note that $ \int _{\partial _R}dh=-2\pi i\, B$
 is purely imaginary since $B\in \R $ (here
$\partial _R$ is oriented as boundary of $D(\infty ,R)$). Therefore,
 the period problem for $(g,dh)$ given by equation~(\ref{eq:periodproblem})
 reduces to the horizontal component.
To study this horizontal period problem, we first compute
$\int_{\partial _R}g\, dh$ as the residue of a meromorphic one-form
in $\{ |z|\leq R\} $, obtaining
    \begin{equation}
\label{eq:pergdh}
\int_{\partial _R}g\, dh=2\pi i t \left(
A-B+i AB\right).
\end{equation}
Arguing analogously with $\int _{\partial _R}\frac{dh}{g}$ we have
\begin{equation}
\label{eq:per1/gdh}
\int_{\partial _R}\frac{dh}{g}=-2\pi i\frac{A+B}{te^{iA}}.
\end{equation}
Hence the period condition (\ref{eq:periodproblem}) reduces to the equation
\begin{equation}
\label{eq:per}
t \left[ A(1+i\, B)-B\right] =\frac{\overline{A}+B}{t}e^{i\, \overline{A}}.
\end{equation}

Next we will show that the period problem (\ref{eq:per}) can be
solved in the parameters $t,A,B$ with arbitrary values of the flux
vector different from vertical (the case of vertical nonzero flux
will be treated in Case~2 below, with a different choice of $g,dh$).

Let $A=x+iy$ with $x,y\in \R $, and rewrite
equation~(\ref{eq:per}) in the following two ways:
\begin{equation}
\label{e1}
t^2e^{-y}[A(1+i B)-B]=(\overline{A}+B)e^{ix}
\end{equation}
\begin{equation}
\label{e2}
 t^2e^{-y}[(x-By-B)+i(Bx+y)]=[(x+B)-iy]e^{ix}.
\end{equation}
We fix $B\geq 0$ 
and for $y\in \R$, define the
following related $\C$-valued curves:
\begin{equation}
\label{eq:L(x)R(x)}
L(x)=(x-By-B)+i(Bx+y),\qquad R(x)=[(x+B)-iy]e^{ix},
\end{equation}
that contain the information on the arguments of the complex numbers
of the left- and right-hand sides of equation~(\ref{e2}), provided
they are nonzero.
Observe that these curves depend on the parameter $y$.

Also note that
\begin{equation}
\label{eq:|L|}
|L(x)|^2=(1+B^2)(x^2+y^2)+B^2-2Bx+2B^2y.
\end{equation}
As $(1+B^2)y^2+2B^2y\geq -\frac{B^4}{1+B^2}$ for all $y\in \R$, then
(\ref{eq:|L|}) gives
\begin{equation}
\label{eq:|L|a}
|L(x)|^2\geq (1+B^2)x^2+B^2-2Bx-\frac{B^4}{1+B^2}=(1+B^2)
\left( x-\frac{B}{1+B^2}\right) ^2.
\end{equation}
Since our goal is to prove that both $t>0$, $A\in \C -\{ 0\} $ can be
chosen so that the period problem (\ref{eq:per}) is solved for a
given $B\in \R$, then in the sequel we will fix $B\geq 0$
and take $x=\mbox{Re}(A)$ larger 
than some particular values; namely,
\begin{equation}
\label{eqx}
 x>\frac{B}{1+B^2}.
\end{equation}
Thus, (\ref{eq:|L|a}) gives that $L(x)\neq 0$ when $x$ satisfies (\ref{eqx}), and so, we can write $L(x)$
in polar coordinates as
\[
L(x)=r_L(x)e^{i\theta_L(x)},
\]
where $r_L(x)$ is the modulus of $L(x)$ and $\theta_L (x)$ is its argument,
taking values in the interval $(-\pi,\pi ]$. 
Note that 
$L(x)$ cannot lie in the closed third quadrant for any value of $x>\frac{B}{1+B^2}$
(if $B=0$, then $L(x)=x+iy$ with $x>0$ and this property is clear; while if
$B>0$, then either $L(x)$ has positive real part in which case $L(x)$ cannot
lie in the closed third quadrant, or $x-By-B\leq 0$ and in this case $Bx+y\geq
Bx+\frac{1}{B}(x-B)=\frac{B^2+1}{B}x-1>0$ so $L(x)$ also does not lie in
the third quadrant), i.e.
 $\theta _L(x)\in (-\frac{\pi }{2},\pi )$ for all $x>\frac{B}{1+B^2}$. 

Viewing $L(x)$ as a curve in $\R^2$ with parameter
$x$ satisfying (\ref{eqx}), we have:
\begin{equation}
\label{|thL'|}
\left|\frac{d \theta_L}{dx}(x)\right|\leq \frac{|L'(x)|}{|L(x)|}\stackrel{(\ref{eq:L(x)R(x)})}{=}\frac{|1+iB|}{|L(x)|}
\stackrel{(\ref{eq:|L|a})}{\leq}
\left| x-\frac{B}{1+B^2}\right| ^{-1}.
\end{equation}

On the other hand, note that $R(x)=0$ if and only if $x=-B$ and
$y=0$. Since we are assuming that $x$ satisfies (\ref{eqx}),
then $R(x)$ cannot vanish so we can we write
\[
R(x)=r_R(x)e^{i\theta_R(x)}
\]
 in polar coordinates, where
$\theta_R(x)=x+\theta_1(x)$ and $\theta_1(x)=\arg [(x+B)-iy]$, which
 lies in the interval $\left(-\frac{\pi}{2}, \frac{\pi}{2}\right)$.
 Therefore,
\begin{equation}
\label{thetaR'}
\frac{d\theta_R}{dx}(x)=
1+\frac{d \theta_1}{d x}(x)\geq 1+\min _{y\in \R }\frac{y}{(B+x)^2+y^2}=
1-\frac{1}{2(B+x)}.
\end{equation}

Comparing the expression $m=m(B,x)=1-\frac{1}{2(B+x)}$ in the
right-hand side of (\ref{|thL'|}) with the one at the right-hand side of
(\ref{thetaR'}), we deduce directly the following property.
\begin{assertion}
\label{ass6.1} Given $B\in [0,\infty )$, there exists
$x_0=x_0(B)$ satisfying (\ref{eqx}) such that
such that if $x\in (x_0,\infty )$, 
then $\frac{d\t _L}{dx}(x)
<\frac{d\t _R}{d x}(x)$.
\end{assertion}

We are now ready to find $A\in \C $ (given $B\in [0,\infty )$) such that the
complex numbers $L(x),R(x)$ in equation (\ref{eq:per}) have the same
argument. 
\begin{assertion}
\label{ass1}
Given $B\geq 0$, we have one of the following two possibilities.
\begin{enumerate}[(A)]
\item If $\theta_R (x_0)\leq \theta_L(x_0)$  (here $x_0=x_0(B)$ refers to
the value obtained in Assertion~\ref{ass6.1}),
then there exists a unique point $x_1=x_1(B,y)\in I=\left[ x_0, x_0+\frac{\pi -\t _R(x_0)}{m(B,x_0)}\right) $
such that $\theta_R (x_1)=\theta_L(x_1)$.
\item If $\theta_R (x_0)\geq \theta_L(x_0)$,
then there exists a unique point $x_1=x_1(B,y)\in I=\left( x_0-\frac{\t _R(x_0)+
\frac{\pi }{2}}{m(B,x_0)},x_0\right] $ such that $\theta_R (x_1)=\theta_L(x_1)$.
\end{enumerate}
Furthermore, the function $(B,y)\mapsto x_1(B,y)$ is continuous.
\end{assertion}
\begin{proof}
Assume we are in case A, i.e., $\t _R(x_0)\leq \t _L(x_0)$.
First note that by~(\ref{thetaR'}), the parameterized curve $\G_R$
given by $x\in [x_0,+\infty )\mapsto (x,\t _R(x))$ lies entirely
above the half-line $r$ that starts at $(x_0,\t _R(x_0))$ with slope $m(B,x_0)$
(this slope can be assumed arbitrarily close to 1 by taking $x_0$ large enough;
note also that  $\theta_R$ restricted to the half-open interval
$[x_0,+\infty )$ is injective).
By comparison of $\G _R$ with $r$, we easily conclude that
$\t _R$ must reach the value $\pi $ at some point $\wt{x}_1>x_0$
which is not greater than $x_0+\frac{\pi -\t _R(x_0)}{m(B,x_0)}$, see
Figure~\ref{fig6} left.
\begin{figure}
\begin{center}
\includegraphics[height=5cm]{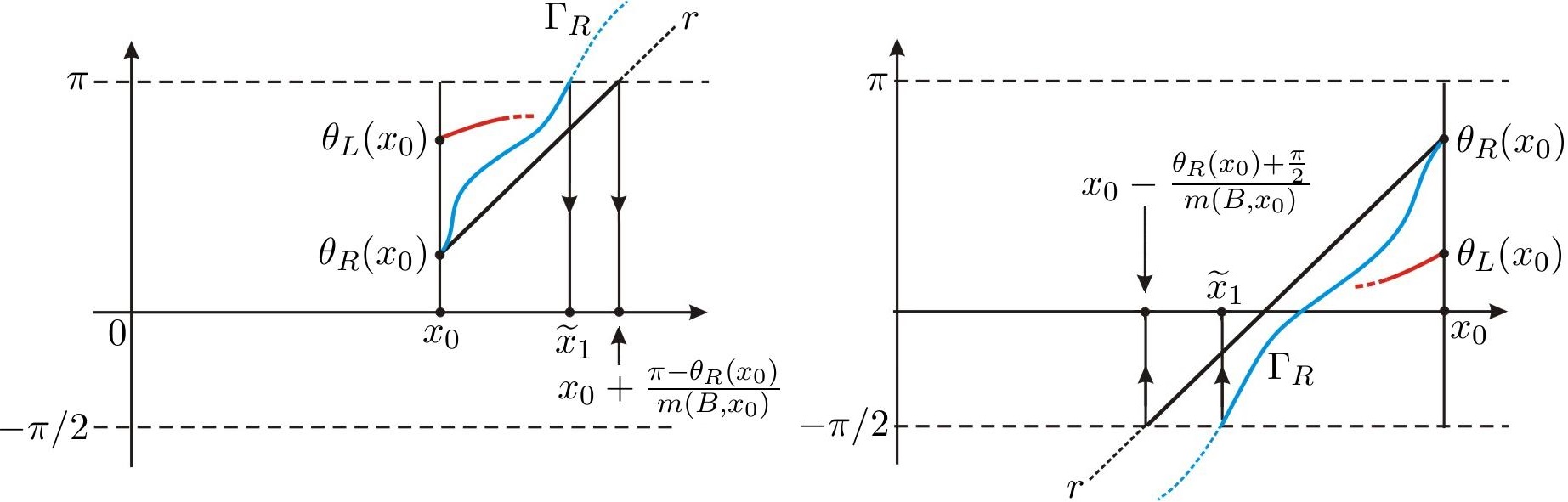}
\caption{Left: Case A of Assertion~\ref{ass1}; the half-line
$r$ has slope $m(B,x_0)$ very close to $1$. The red and blue
curves must intersect. Right: Case B of Assertion~\ref{ass1}.
}
\label{fig6}
\end{center}
\end{figure}
Since the range of $x\in [x_0,\infty )\mapsto \t _L(x)$ is $(-\frac{\pi }{2},\pi )$,
the intermediate value theorem applied to the function $\t _R-\t _L$ implies that
there is at least one point $x_1\in [x_0,\wt{x}_1)\subset I$ such that $\t _R(x_1)=
\t _L(x_1)$. Since $(\t _R-\t _L)'(x)>0$ for all $x>x_0$ by Assertion~\ref{ass6.1},
then the point $x_1=x_1(B,y)$ is unique.

Case B can be proved with similar reasoning as in the last paragraph,
using Figure~\ref{fig6} right instead of Figure~\ref{fig6} left (note that we need to apply Assertion~\ref{ass6.1}
for values of $x$ less than $x_0$, which can be done by taking $x_0$ large enough).
Finally, the continuous dependence of
$x_1$ with respect to $B,y$ follows from the same type of dependence for the data
$\t _L,\t_R,x_0,m(B,x_0)$. This finishes the proof of the assertion.
\end{proof}
Fix $B\in [0,\infty )$. Given $y\in \R$, consider the value $x_1=x_1(B,y)$
appearing in Assertion~\ref{ass1} 
and let $A=x_1+iy$.
Since the left- and right-hand sides of equation~(\ref{e2})  (with
$x=x_1$, note that $t>0$ is still to be determined) are not zero and the arguments
of the complex numbers on each of the sides are the same, then there
exists a unique $t=t(B,y)\in(0,\infty)$ varying continuously in
$B,y$ so that equation~(\ref{e2}) holds. Therefore, the period
problem (\ref{eq:per}) is solved for the Weierstrass data
(\ref{eq:gdh1}) given by the values $B, A=x_1+iy$, $t$. Let $E(B,y)$
denote the related minimally immersed annulus in $\rth$ defined by
these Weierstrass data. Let $F(B,y)$ denote the length of the
horizontal flux of $E(B,y)$, which by equation (\ref{eq:flux}) can be identified with the
integral
\begin{equation}
\label{lengthflux}
F(B,y)=\left| \int_{\partial_R}g\,dh\right| \stackrel{(\ref{eq:pergdh})}{=}
2\pi t\left| (x_1-By-B)+(Bx_1+y)i\right| .
\end{equation}
By equation~(\ref{e2}) and the fact that $y\in \R \mapsto x_1(B,y)$
is bounded for $B$ fixed (by Assertion~\ref{ass1}), 
we see that $y\in \R \mapsto t(B,y)$
grows exponentially to $\infty$ as $y\to +\infty$ and
decays exponentially to $0$ as $y\to -\infty$. Hence,
from (\ref{lengthflux}) we also find:
\[
\lim_{y\to+\infty} F(B,y)=\infty, \qquad \lim_{y\to -\infty} F(B,y)=0.
\]
Since $y\in \R \mapsto F(B,y)$ is continuous, the
intermediate value theorem proves the next assertion.

\begin{assertion}
\label{ass6.4} 
Given $B\in [0,\infty)$, the minimally immersed annuli $\{
E(B,y)\mid y\in \R \} $ defined above attain all possible positive
lengths for the horizontal component of their fluxes (the vertical
component of their flux is always equal to $-2\pi B$). Hence, by the
axiom of choice, for each $a>0$ and $b\in [0,\infty )$, there exists
$y=y(a,b)\in \R$ such that $E(\frac{b}{2\pi},y(a,b))$ has vertical
component of its flux equal to $-b$ and the length of the horizontal
component of its flux is equal to $a$. After a rotation of
$E(\frac{b}{2\pi},y(a,b))$ around the $x_3$-axis, we obtain the
desired example $E^{a,b}$ with flux vector $(a,0,-b)$.
\end{assertion}

\par
\vspace{.2cm} \noindent
 {\bf Case 2: The flux vector is vertical.}
 \par
We now complete the construction of the canonical examples
$E^{0,b}$, as complete immersed minimal annuli with infinite total
curvature and flux vector $(0,0,-b)$ for any $b\in (0,\infty )$. Fix
$B\in (0,\infty )$ and consider the following choices for $g, dh$:
\begin{equation}
\label{eq:gdh2} g(z)=e^{iz}\frac{z-A}{z-\overline{A}},\qquad dh=
\left( 1+\frac{B}{z}\right) dz, \quad z\in D(\infty, R),
\end{equation}
where $A\in \C-\{0\}$ and $R> |A|$. Our goal is to show that given
$B>0$, there exists $A$ as before so that the
Weierstrass data given by (\ref{eq:gdh2}) solve the period problem
(\ref{eq:periodproblem}) and define a complete immersed minimal
annulus with infinite total curvature and flux vector $(0,0,-2\pi
B)$. Similarly as in Case 1, we can write $g$ as $g(z)=e^{iz+f(z)}$
where $f(z)=\log \frac{z-A}{z-\overline{A}}$, which is univalent and
holomorphic in $D(\infty ,R)$ because $R>|A|$, and $f$ extends to $z=\infty $ 
with $f(\infty )=0$. The period problem
(\ref{eq:periodproblem}) for $(g,dh)$ reduces in this vertical flux
case to
\begin{equation}
\label{eqw0} \int_{\partial_R}g \,
dh=\int_{\partial_R}\frac{dh}{g}=0.
\end{equation}

We first compute $\int_{\partial_R}g\,dh$ in terms of residues of a
meromorphic one-form in $\{|z|\leq R\}$ which we impose to be zero:
\begin{equation}
\label{eqw1new}
 \int_{\partial_R} g\,dh =-2\pi i \left[ e^{i\overline{A}}\left(
\overline{A}-A+B-\frac{AB}{\overline{A}}\right)
+\frac{AB}{\overline{A}}\right] =0.
\end{equation}
A direct computation shows that the integral
$\int_{\partial_R}\frac{dh}{g}$ is the complex conjugate of the
expression involving $A,B$ in equation~(\ref{eqw1new}). Hence it
suffices to find, given $B>0$, a nonzero complex number
$A$ solving~(\ref{eqw1new}).

Multiplying equation~(\ref{eqw1new}) by  $\overline{A}$ and dividing
by $-2\pi i$ yields:
\begin{equation}
\label{eqw2new} e^{i\overline{A}}\left(
\overline{A}^2-|A|^2+B(\overline{A}-A)\right) =-AB \quad
\end{equation}
Letting $A=x+iy$, we obtain
\begin{equation}
 \label{eqw3new}
2ye^y e^{ix}\left[ y+i(x+B)\right] = B\left( x+iy \right) .
\end{equation}
After fixing $y>0$ and letting $x$ vary in the range $\{ 1<x<\infty
\} $, 
we have that $1<\frac{d\theta_L}{dx}(x)$
and $\frac{d \theta_R}{dx}(x)< 0$, where $\t_{L}(x)$ and $\t_R(x)$
denote the argument of the left-
and right-hand sides of equation~(\ref{eqw3new}), both considered in the
variable $x$. Also note that we
can take branches of $\t _R(x),\t _L(x)$ so that
$0<\t _R(x)<\frac{\pi}{2}$ for the range of
values of $x$ and $y$ that we are considering and $\t _L(2\pi
-\frac{\pi }{2})\in (-\frac{\pi}{2},0)$.

We define the interval $J= (2\pi -\frac{\pi}{2}, 2\pi +\frac{\pi}{2}) $.
It follows that given $B>0$, for each $y>0$ there is a
unique $x =x(y)\in J$ such that the arguments of both sides
of~(\ref{eqw3new}) are equal for $A=x(y) +iy$. Note that the norm of
the left-hand side of~(\ref{eqw3new}), considered now to be a
positive function of $y$, varies continuously with limit value $0$
as $y$ tends to $0^+$ and $+\infty$ as $y\to \infty$, and this norm
grows exponentially to $+\infty$ as $y\to \infty$. On
the other hand, the norm of the right-hand side of~(\ref{eqw3new})
is bounded away from zero and grows linearly in $y$ as
$y$ tends to $+\infty$.  Thus, for any initial value $B>0$, 
there exists $y\in \R^+$ and the corresponding $x=x(y)\in J$ such
that the left- and right-hand sides of~(\ref{eqw3new}) are equal.
This completes the proof of Case~2.

We can summarize this Case 2 in the next assertion. The proof of the
last statement of this assertion follows from the observation that
when the flux vector is $(0,0,-b)$, then the Weierstrass data of the
corresponding immersed minimal annulus $E^{0,b}$ is given by
equation~(\ref{eq:gdh2}) with $B=\frac{b}{2\pi }$. It is easy to
check that the conformal map $z \stackrel{\Phi }{\mapsto
}\overline{z}$ in the parameter domain $D(R,\infty)$ of $E^{0,b}$
satisfies $g\circ \Phi =1/\overline{g}$, $\Phi ^*dh=\overline{dh}$.
Hence, after translating the surface so that the image of the point
$R\in D(R,\infty )$ lies on the $x_3$-axis, we deduce that $\Phi $
produces an isometry of $E^{0,b}$ which extends to a 180$^o$-rotation
of $\rth$ around the $x_3$-axis.

\begin{assertion}
\label{ass6.6}
 For any $b>0$, there exists a constant $A\in \{x+iy
\mid x\in J, \; y\in \R^+ \} $ depending on $b$ such that the
corresponding Weierstrass data in (\ref{eq:gdh2}) (with
$B=\frac{b}{2\pi }$) define a minimally immersed annulus $E^{0,b}$
with flux vector $(0,0,-b)$. Furthermore, the conformal map $z\to
\overline{z}$ in the parameter domain $D(R,\infty)$ of $E^{0,b}$
extends to an isometry of $\rth$ which is a 180$^o$-rotation around the
$x_3$-axis.
\end{assertion}
Finally, we can join Cases 1 and 2 to conclude the following main result of this
section.
\begin{assertion}
\label{ass6.7} For each $a,b\geq 0$, there exists a
complete, minimally immersed annulus $E^{a,b}$ defined on some
$D(\infty,R)$, $R=R(a,b)>0$ by the Weierstrass data
\[
\left( g(z)=e^{iz+f(z)},\ dh=\left( 1+\frac{1}{2\pi }\frac{b}{z-\mu }\right) dz\right) ,
\]
where $f(z)$ is a holomorphic function in $D(\infty ,R)\cup
\{ \infty \}  $ with $f(\infty)=0$, and $\mu \in \C $.
Furthermore, the flux vector of $E^{a,b}$ along $\{ |z|=R\} $
is $(a,0,-b)$ and when $a=b=0$, we choose $f(z)=0$.
\end{assertion}

\section{The proof of Theorem~\ref{thm1.3} and the embeddedness
of the canonical examples $E_{a,b}$.}
\label{sec7}
In Section~\ref{sec6} we focused on the Weierstrass representation of the
surfaces under consideration.  The actual surfaces themselves
for given Weierstrass data are only determined up
to a translation, because they depend on integration of analytic
forms after making a choice of a base point $z_0$ in
the parameter domain ${\cal D}\subset \C$.  The  parametrization
of the surfaces is calculated as a path integral where
the path begins at the base point.  In the proof of the next theorem,
the reader should remember that different choices
of base points give rise to translations of the image minimal annulus that we
are considering; the image point of the base point is always the origin $\vec{0}\in \R^3$.

Theorem~\ref{thm1.3} stated in the Introduction follows from the
material described in the last section together with items~{1, 5, 6,
7, 8} and { 9} of the next theorem. Theorem~\ref{th1.1} is also
implied by the next statement. Recall that the remaining item~3c
to be proved in Theorem~\ref{th1.1} follows immediately from
items~\ref{it8}, \ref{it9} below and the fact that $E_{0,0}$ can be taken to
be the end of the vertical helicoid with Weierstrass data given by
$g(z)=e^{iz}$, $dh=dz$. In item~\ref{it6} below, we use the notation
introduced just before the statement of Theorem~\ref{thm1.3}.

The
reader will find it useful to refer to Figure~\ref{Eab} in the Introduction for a
visual geometric interpretation of some qualities of the complete,
embedded minimal annulus $E$ appearing in the next statement.
\begin{theorem}
\label{main} Suppose $E$ is a minimally immersed annulus in $\rth$
which is the image of a conformal immersion $X\colon D(\infty,
R')\to \rth$, with flux vector $(a,0,-2\pi \l )$ for some $a,\l \geq0$,
whose Weierstrass data $(g,dh)$ on $D(\infty ,R')$
are given by
\[
\left( g(z)=e^{iz+f(z)},\ dh=\left( 1+\frac{\l }{z-\mu }\right) dz\right) ,
\]
where $f(z)$ is a holomorphic  at
$\infty$ with $f(\infty)=0$, and $\mu \in \C$, $|\mu |<R'$.
Then:
\begin{enumerate}[1.]
\item \label{it1} For some $R_1\geq R'$, $X$ restricted to
$D(\infty, R_1)$ is injective.
In particular, each canonical example $E^{a,b}$ described in
Assertions~\ref{ass6.4} and~\ref{ass6.6} contains an end representative 
$E_{a,b}$ which is a properly embedded minimal annulus with compact
boundary and infinite total curvature.

\item  \label{it2} The Gaussian curvature function $K$ of $X$ satisfies
$\limsup_{R\to \infty} \left| K|_{X(D(\infty ,R))}\right| =1$.
In particular, $K$ is bounded.
\item  \label{it3,4} There exist $(x_T, y_T), (x_B,y_B)\in \R^2$ such that
the image curves $X([R',\infty)), X((-\infty, -R'])$ are asymptotic
to the vertical half-lines $r_T=\{(x_T,y_T,t)\mid t\in [0,\infty)\}$,
$r_B=\{(x_B, y_B,t)\mid t\in (-\infty,0]\}$ respectively.
\item
\label{it5}
$(x_T,y_T)-(x_B,y_B)=(0, -\frac{a}{2})$.
\item
\label{it7} The sequences of translated surfaces
$X_n=E-(0,0,2\pi n +\l \log n)$ and
$Y_n=E+(0,0,2\pi n-\l \log n)$ converge respectively to
right-handed vertical helicoids $H_T$, $H_B$ where $r_T$ is
contained in the axis of $H_T$ and $r_B$ is contained in the axis of
$H_B$. Furthermore, $H_B=H_T+(0, \frac{a}{2},0)$.
\item
\label{it6} After a fixed  translation of $E$, assume that
\begin{equation}
\label{norm}
(x_T,y_T)+(x_B,y_B)=(0,0)\quad \mbox{and}\quad x_3(z)=x+\l \log
|z-\mu|,
 \end{equation}
where $x_3$ is the third coordinate function of $X$ and
$z=x+iy$.
\begin{enumerate}[a.]
\item \label{it6-1}There exists an $R_E>1$ such that
$E-C(R_E)$ consists of two disjoint multivalued graphs $\Sigma _1, \Sigma
_2$ over $D(\infty ,R_E)\subset \R^2$ of smooth functions
$u_1,u_2\colon \widetilde{D}(\infty ,R_E)\to \R $ whose gradients with respect
to the metric on $\widetilde{D}(\infty ,R_E)$ obtained by pulling back the
standard flat metric in $D(\infty ,R_E)$,
satisfy $\nabla u_i(r,\theta )\to 0$ as $r\to \infty $.

\item \label{it6-2} Consider the
multivalued graphs $v_1,v_2\colon \widetilde{D}(\infty ,R_E)\to \R $
defined by
 \begin{eqnarray}
v_1(r,\theta )=\rule{0cm}{.5cm}\left( \t +\frac{\pi}{2}\right) +
\l \log
\sqrt{ \left( \t+\frac{\pi}{2}\right) ^2 + [\log(2r)]^2}, \label{ss}
 \\
v_2(r,\theta )=\rule{0cm}{.5cm}\left( \t-\frac{\pi}{2}\right)
+\l \log \sqrt{\left( \t-\frac{\pi}{2}\right) ^2 +
[\log(2r)]^2}. \label{sss}
 \end{eqnarray}
Then, for each $n\in \N$, there exists an $R_n >R_E$ such that $|u_i -v_i| <\frac1n$
in $\widetilde{D}(\infty ,R_n)$. In particular,
for $\theta$ fixed and $i=1,2$, we have $\lim_{r\to \infty}
\frac{u_i(r,\theta)}{\log (\log(r))} =\l $ (see Figure~\ref{Eab}
and note that $b=2\pi \l $ 
there).

\item \label{it6-3}
 Furthermore:
\begin{enumerate}[I.]
 \item  When $a=0$, then on $\widetilde{D}(\infty ,R_E)$ we have $|u_i -v_i| \to 0$ as
$r+|\theta |\to\infty $.
\item The separation function $w(r,\theta )=u_1(r,\theta )-u_2(r,\theta )$
converges to $\pi $ as $r+|\theta |\to
\infty $.
\end{enumerate}
\end{enumerate}

\item
\label{it8}
 Suppose that $X_2\colon D(\infty, R_2)\to \rth$ is
another conformal minimal immersion with the same flux vector
$(a,0,-2\pi \l )$ as $X$ and Weierstrass data $(g_2,dh_2)$ given by
\[
\left( g_2(z)=e^{iz+f_2(z)}, dh_2=\left( 1+\frac{\l }{z-\mu_2}\right) dz\right)
\]
where $f_2(z)$ is holomorphic  at $\infty$ with $f_2(\infty)=0$, and $\mu _2\in \C$.
Then, there exists a  vector $\tau \in \R^3$ such that the
$X_2(D(\infty, R_2)) +\tau$ is asymptotic to $E=X(D(\infty, R_1))$.
In particular, for some translation vector $\wh{\tau}\in
\R^3$, $E+\wh{\tau}$ is asymptotic to the canonical
example $E_{a,b}$, where $b=2\pi \l $. Note that if $E$ and $E_{a,b}$
are each normalized by a translation as in (\ref{norm}), then $E$ is asymptotic to $E_{a,b}$.
\item
\label{it9}
Given $(a,b)\neq(a',b') \in [0,\infty ) \times [0,\infty )$,
then, after any homothety and a rigid motion applied to $E_{a,b}$,
the image surface is not asymptotic to $E_{a',b'}$.
\item
\label{it10}
Each of the canonical examples $E_{0,b}$ for $b\in \R$ is invariant
under the {\rm 180$^o$}-rotation around the $x_3$-axis $l$,
and $l\cap E_{0,b}$ contains two infinite rays.
\end{enumerate}
\end{theorem}

\begin{proof}
We will follow the ordering
\ref{it2}$\to $\ref{it3,4}$\to $\ref{it5}$\to $\ref{it7}$\to $\ref{it6}$\to $\ref{it8}$\to $\ref{it9}$\to $\ref{it10}$\to $\ref{it1}
when proving the items in the statement
of Theorem~\ref{main}.

As $f$ is holomorphic in a neighborhood of $\infty $ with  $f(\infty )=0$,
the series expansion of $f$ has only negative powers of $z$, so we can
find $C>0$ depending only on $f$ such that
\begin{equation}
|f(z)|\leq C|z|^{-1},\qquad |f'(z)|\leq C|z|^{-2},\qquad
\qquad \mbox{for any } z\in D(\infty ,R').
\label{eq:estimf}
\end{equation}
Throughout the proof, we will assume that $|C|\geq R'$.  We will frequently
refer to these two simple estimates for $f(z)$ and $f'(z)$.

We first prove item~\ref{it2}. The expression of the absolute Gaussian
curvature $|K|$ of $X$ in terms of its Weierstrass data is
\[
|K|=\frac{16}{(|g|+|g|^{-1})^4}\frac{|dg/g|^2}{|dh |^2 }=\frac{1}{[\cosh
\mbox{Re} (iz+f(z)]^4}
\frac{\left|i+f'(z)\right|^2}{\left| 1+\frac{\l }{z-\mu }\right|^2}.
\]
Since $f$ is holomorphic in a neighborhood of $\infty $ with $f(\infty )=0$,
the last right-hand side
 is bounded from above, which easily implies item~\ref{it2}.

We next demonstrate item~\ref{it3,4}. Let $\g (r)=r$, $r\geq R'$,
be a parametrization of the portion of the
positive real axis in $D(\infty ,R')$. We will check that $X\circ \g $
is asymptotic to a vertical half-line pointing up, which will be
$r_T$; the case of $r_B$ follows from the same arguments, using
$\g (r)=-r$. From the Weierstrass data of $X$, we have that for
given real numbers $T,R$ with $R'<R<T$:
\[
(x_1+ix_2)(T)-(x_1+ix_2)(R)
=\frac{1}{2}\left( \overline{
\int _{R}^{T}\frac{dh}{g}}-\int _{R}^{T}g\, dh\right)
\]
\[
=\frac{1}{2}\left( \overline{\int _{R}^{T}e^{-it-f(t)}\left(
1+\frac{\l }{t-\mu}\right) dt}-\int _{R}^{T}e^{it+f(t)}\left(
1+\frac{\l }{t-\mu}\right) dt\right)
\]
\begin{equation}
\label{eq:intnew} =\frac{1}{2}\int _{R}^{T}e^{it}\left[
e^{-\overline{f(t)}}\left( 1+\frac{\l }{t-\overline{\mu}}\right)
-e^{f(t)}\left( 1+\frac{\l }{t-\mu}\right) \right] dt.
\end{equation}
Expanding the expression between brackets in the last displayed equation
as a series in the variable $t$, we find
\[
e^{-\overline{f(t)}}\left( 1+\frac{\l }{t-\overline{\mu}}\right)
-e^{f(t)}\left( 1+\frac{\l }{t-\mu}\right) =-\frac{2\mbox{Re}(C_1)}{t}
+{\cal O}(t^{-2}),
\]
where $f(t)=\frac{C_1}{t}+{\cal O}(t^{-2})$, $C_1\in \C-\{ 0\} $
and ${\cal O}(t^{-2})$ denotes a function
such that $t^2{\cal O}(t^{-2})$ is bounded as $t\to \infty $. Hence, (\ref{eq:intnew})
yields
\[
(x_1+ix_2)(T)-(x_1+ix_2)(R)
 =-\mbox{Re} (C_1)\int_{R}^{T}\frac{e^{it}}{t}\, dt
 +\int_{R}^{T}e^{it}{\cal O}(t^{-2})\, dt
 \]

Note that $\left| \int _{R}^{T}e^{it}t^{-2}\, dt\right| \leq \int
_R^Tt^{-2}dt=R^{-1}-T^{-1}$ which converges to $R^{-1}$ as $T\to
+\infty $. On the other hand, taking $R,T$ of the form $R=\pi m$,
$T=\pi n$ for large positive integers $m<n$, we have
\[
\int _{R}^{T}\frac{e^{it}}{t}\, dt=\left( \sum _{k=1}^{n-m}\int _{\pi (m+k-1)}^{\pi (m+k)}\frac{\cos t}{t}dt
\right) +i\left( \sum _{k=1}^{n-m}\int _{\pi (m+k-1)}^{\pi (m+k)}\frac{\sin t}{t}dt\right)
\]
Both parenthesis in the last expression are partial sums of alternating
series of the type $\sum _k(-1)^ka_k$
where $a_k\in \R^+ $ converges monotonically to zero as $k\to \infty $. By the alternating
series test, both series converge
and we conclude that $(x_1+ix_2)(T)-(x_1+ix_2)(R)$ converges to a complex number as $T\to \infty $
(independent of $R$). As the third coordinate function of the immersion satisfies
\[
x_3(T)-x_3(R)=\mbox{Re}\int _R^Tdh=T-R+\l \log \left| \frac{T-\mu}{R-\mu }\right|,
\]
which tends to $\infty $ as $T\to \infty $, then we conclude that item~\ref{it3,4} of the theorem holds.

We next prove item~\ref{it5}, which is equivalent to:
\begin{equation}
\label{eqf}
\lim_{r\to +\infty} \left[(x_1+i x_2)(r)-(x_1+ix_2)(-r)\right]
=-\frac{ia}{2}.
\end{equation}
To see that (\ref{eqf}) holds, we  use computations inspired by those
in Section 5 of~\cite{hkp1}. Take
$r>R'$ and consider the boundary $\partial S_r$ of the square
$S_r \subset \C$ with vertices $r+ir,r-ir, -r-ir,-r+ir$.
$\partial S_r$ can be written as $\Lambda _r-\overline{\Lambda _r}$,
where $\Lambda _r$ is the polygonal arc with vertices $-r,-r+ir,r+ir,r$
(the orientations on $\partial S_r$ and $\Lambda _r$ are chosen so that
both lists of vertices are naturally well-ordered, see Figure~\ref{figsquare}).
\begin{figure}
\begin{center}
\includegraphics[width=6cm]{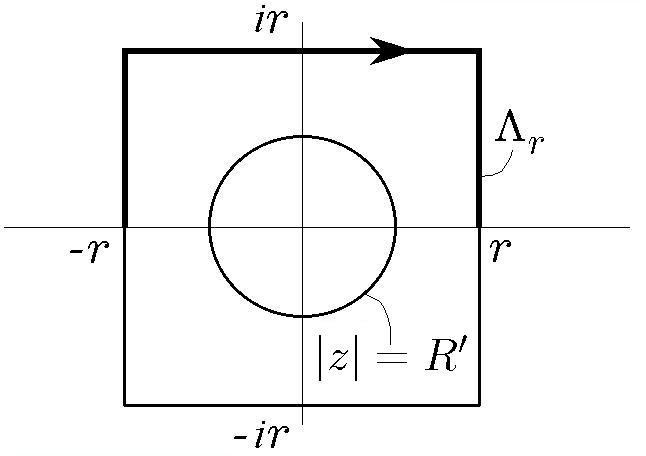}
\caption{The square $\Lambda _2-\overline{\L _r}$ is homologous to
the circle $\{ |z|=R'\} $ in $D(\infty ,R')$. }
\label{figsquare}
\end{center}
\end{figure}
As $\{ |z|=R'\} $ and $\partial S_r$ bound a compact domain in the
surface, it follows that the horizontal component of the flux of $X$
along $\partial S_r$ equals $(a,0)$. By equations~(\ref{eq:periodproblem}) and (\ref{eq:flux}),
this is equivalent to the following equality:
\begin{equation}
\int _{\Lambda_r}\frac{dh}{g}=\int
_{\overline{\Lambda_r}}\frac{dh}{g} +ia. \label{eq:flux0}
\end{equation}
Hence,
\[
(x_1+ix_2)(r)-(x_1+ix_2)(-r)=\int _{\Lambda_r}d(x_1+ix_2)
=\frac{1}{2}\left( \overline{\int _{\Lambda_r}\frac{dh}{g}}-\int
_{\Lambda _r}g\, dh\right)
\]
\[
\stackrel{(\ref{eq:flux0})}{=}
 \frac{1}{2}\left( \overline{\int _{\overline{\Lambda _r}}\frac{dh}{g}}-ia-\int _{\Lambda
_r}g\, dh\right) =\frac{1}{2}\left( I_2-I_1\right) +
\frac{1}{2}\left( \wh{I}_2-\wh{I}_1\right) -\frac{ia}{2},
\]
where $I_1=\int _{\Lambda _r}g\, dz$,
$I_2=\int _{\overline{\Lambda _r}}\overline{g}^{-1}\,
d\overline{z}$, $\h{I}_1=\int _{\Lambda _r}g\, \frac{\l }{z-\mu }dz$ and
$\wh{I}_2=\int _{\overline{\Lambda _r}}\overline{g}^{-1}\,
\overline{\frac{\l }{z-\mu }dz}$.

We now compute and estimate the norm of each
of these four integrals separately. We first calculate $I_1$ and $I_2$.
\[
\begin{array}{rcl}
I_1&=&{\displaystyle \int _0^rg(-r+iv)i\, dv
+\int _{-r}^rg(u+ir)\, du-\int _0^rg(r+iv)i\, dv}\\
&=&{\displaystyle e^{-r}\int _{-r}^re^{iu+f(u+ir)}\,
du+i\int _0^re^{-v}\left( e^{-ir+f(-r+iv)}-e^{ir+f(r+iv)}\right) dv.}
\end{array}
\]
If we take $r$ such that $\frac{r}{2\pi }$ is integer, then $I_1$
becomes
\[
I_1=e^{-r}\int _{-r}^re^{iu+f(u+ir)}\,
du+i\int _0^re^{-v}\left( e^{f(-r+iv)}-e^{f(r+iv)}\right) dv.
\]
and a similar computation for $I_2$ gives
\[
I_2=e^{-r}\int _{-r}^re^{iu-\overline{f(u-ir)}}du+i\int _0^re^{-v}
\left( e^{-\overline{f(-r-iv)}}-
e^{-\overline{f(r-iv)}}\right) dv.
\]
Thus, both $I_1$ and $I_2$ are sum of two integrals and each one
of these four integrals can be estimated using
(\ref{eq:estimf}) as follows:
\[
\begin{array}{rcl}
{\displaystyle \left| e^{-r}\int _{-r}^re^{iu+f(u+ir)}\, du\right| }&\leq &
{\displaystyle e^{-r}\int _{-r}^re^{\mbox{\footnotesize Re}(f(u+ir))}\, du }\\
& \leq &{\displaystyle e^{-r}\int _{-r}^re^{C|u+ir|^{-1}}\,
du\leq e^{-r}\int _{-r}^re^{Cr^{-1}}
\, du=2re^{-r+Cr^{-1}},}
\end{array}
\]
and analogously, $\left|e^{-r}\int _{-r}^re^{iu-\overline{f(u-ir)}}\,
du\right| \leq 2re^{-r+Cr^{-1}}$, while
\[
\left| i\int _0^re^{-v}\left( e^{f(-r+iv)}-e^{f(r+iv)}\right) dv\right|
\leq \int _0^re^{-v}\left| e^{f(-r+iv)}-e^{f(r+iv)}\right|  dv.
\]

Using (\ref{eq:estimf}), it is straightforward to check that both
$e^{f(-r+iv)}$ and $e^{f(r+iv)}$ can be expressed as
$1+{\cal O}(r^{-1})$, hence the last expression is of the type
$\int _0^re^{-v}{\cal O}(r^{-1})\, dv\leq C_1r^{-1}
\left( 1-e^{-r}\right) $,  for a constant $C_1>0$ independent of
$r$. Similarly, $\left| i\int _0^re^{-v}
\left( e^{-\overline{f(-r-iv)}}-e^{-\overline{f(r-iv)}}\right) dv\right|
\leq C_1r^{-1}\left( 1-e^{-r}\right) $. Since the integrands of
$\wh{I}_1$, $\wh{I_2}$ differ from the respective integrand of $I_1$,
$I_2$ by a product with $\frac{\l }{z-\mu }$ and since $|\frac{\l }{z-\mu }| < 1$
for $|z|$ sufficiently large, then $|\wh{I_1}|<|I_1|$,
$|\wh{I_2}|<|I_2|$, when $r$ is large.
In summary,
\[
\left| (x_1+ix_2)(r)-(x_1+ix_2)(-r)+\frac{ia}{2}\right|
\leq |I_2|+|I_1|
\leq 2\left[ 2re^{-r+Cr^{-1}}+C_1r^{-1}\left( 1-e^{-r}\right) \right] ,
\]
which tends to zero as $r\rightarrow +\infty $.
This completes the proof that equation~(\ref{eqf}) holds
and so proves item~\ref{it5} of the theorem.

Next we prove item~\ref{it7}. For $n\in \Z$, $2\pi n>R'$, let
$D(n)=\{z\in \C \mid |z-2\pi n|\leq 2\pi n-R'\}$ be the closed disk of
radius $2\pi n-R'$ centered at $2\pi n$, and note that $D(n)\subset
D(\infty ,R')$. For $|n|$ large, consider the immersion
$\wt{X}_n\colon D(n) \to \rth$ with the same Weierstrass data as $X$
restricted to the disk $D(n)$, and with base point $2\pi n$ (in
particular, $\wt{X}_n$ is a translation of a portion of $X$). The
minimal immersions ${\wt{X}}_n \colon D(n)\to \rth$ converge
uniformly to the right-handed vertical helicoid $H$ with Weierstrass
representation data $g(z)=e^{iz}$, $dh=dz$, $z\in \C $, and base
point $\wt{z}_0=0$. After letting $z_0\in D(\infty ,R')$ denote the
base point of $X$, then the third coordinate functions $x_3$ of $X$
and $\wt{x}_{3,n}$ of $\wt{X}_n$ are related by
\[
\begin{array}{rcl}
x_3(z)-2\pi n+\mbox{Re}(z_0)&=&
{\displaystyle \mbox{Re} \int _{z_0}^z\left( 1+\frac{\l }{z-\mu }\right) dz
 -2\pi n+\mbox{Re} (z_0)}
\\
&=&{\displaystyle \mbox{Re} (z)+\l \log \left| \frac{z-\mu }{z_0-\mu }\right|  -2\pi n}
\\
&=&{\displaystyle
 \int _{2\pi n}^z\left( 1+\frac{\l }{z-\mu }\right) dz+
\l \log \left| \frac{2\pi n-\mu }{z_0-\mu }\right|
}
\\
&=&{\displaystyle
\wt{x}_{3,n}(z)+\l \log \left| \frac{2\pi n-\mu }{z_0-\mu }\right| }
\\
&=&{\displaystyle
\wt{x}_{3,n}(z)+\l \log n+\l \log \left| \frac{2\pi -\frac{\mu }{n}}{z_0-\mu }\right| }.
\end{array}
\]
Therefore, the minimal surfaces $X_n:=E+(0,0,-2\pi n-\l \log n)$
converge as $n\to \infty $ to a translated image $H_T:=H+\tau $ of
the helicoid $H$, where $\tau \in \R^3$ has vertical component
$-\mbox{Re}(z_0)+\l \log \frac{2\pi}{|z_0-\mu |}$. Clearly, $r_T\subset r_T+(0,0,-2\pi n-\l \log n)$
is contained in $H_T$. Similar computations give that the surfaces
$E+(0,0,2\pi n-\l \log n)$ converge as $n\to \infty $ to a
translated image $H_B:=H+\tau '$ of $H$ such that $x_3(\tau
)=x_3(\tau ')$. Now item~\ref{it5} gives that $\tau -\tau
'=(0,-\frac{a}{2},0)$, and item~\ref{it7} holds.

Next we prove item~\ref{it6}. Throughout the proof of this item we
will let 
\[
\Delta (n)=\{ z=x+iy\in \C \mid |y|\geq n\} ,
\]
where $n\in \N$. First translate
$E$ as required in the first sentence of item~\ref{it6} of Theorem~\ref{main},
and use the same notation $X$ for the translated immersion
which parameterizes $E$. Next observe that for some fixed large
positive integer $n_0$ and for each $n\in \N$ with $n\geq n_0$, we
have
\begin{enumerate}[(P1)]
 \item $\partial D(\infty, R')$ is contained in the horizontal strip
 $\C -\Delta (n)$.
\item As $n\to \infty$, the Gaussian image of $X(\Delta(n))$ is contained in smaller
and smaller neighborhoods of $(0,0,\pm 1)$ in $\esf^2$, which
converge as sets to the set $\{(0,0,\pm1)\}$ in the limit.
 \end{enumerate}
Property (P2) implies that $E$ is locally graphical over its projection to the $(x_1,x_2)$-plane. We next study
this locally graphical structure.

From the Weierstrass representation of $X$ we conclude that for
$z=x+iy\in D(\infty ,R')$,
\begin{eqnarray}
(x_1+ix_2)(z)&=& \rule{0cm}{.6cm}{\displaystyle \frac{1}{2}\left( \overline{
\int^{z}\frac{dh}{g}}- \int^{z}g\, dh\right) }\nonumber
\\
&=&\rule{0cm}{.6cm}
\frac{1}{2}\left(
\overline{\int ^ze^{-i\xi -f(\xi )}\left( 1+\frac{\l}{\xi -\mu }
\right) d\xi }
-\int ^ze^{i\xi +f(\xi )}\left( 1+\frac{\l}{\xi -\mu }\right) d\xi
\right) \nonumber
\\
 &= &\rule{0cm}{.6cm}
\frac{1}{2}\left(
\overline{\int ^ze^{-i\xi }\left( 1+F_1(\xi )\right) d\xi }-\int ^ze^{i\xi }\left( 1+F_2(\xi )\right) d\xi  \right) \nonumber
\\
&=& \rule{0cm}{.6cm}
\frac{i}{2}(e^{iz}-e^{i\overline{z}})+\frac{1}{2}\left( \overline{A_1(z)}-A_2(z) \right)  \nonumber
\\
&=&\rule{0cm}{.6cm}
\frac{ie^{ix}}{2}(e^{-y}-e^y)+\frac{1}{2}\left( \overline{A_1(z)}-A_2(z) \right)  \label{eq:n},
\end{eqnarray}
where $F_1(\xi ),F_2(\xi )$ are holomorphic functions in $D(\infty ,R')\cup \{ \infty \}$ with $F_i(\infty )=0$, $i=1,2$
(we are using that $f(z)$ is holomorphic at $\infty $ with $f(\infty )=0$), and
\[
A_1(z)=\int ^ze^{-i\xi }F_1(\xi )\, d\xi ,\qquad
A_2(z)=\int ^ze^{i\xi }F_2(\xi )\, d\xi , \qquad |z|\geq R'.
\]

Since $F_1,F_2$ are holomorphic and vanish at infinity, we have
the following technical assertion, which will be proven in the appendix (Section~\ref{appendix}).
\begin{definition}
A complex valued function $E(z)$ defined in $D(\infty ,R')$ is said to satisfy property~$\diamondsuit$
if $E$ is bounded and $|E(z)|$ can be made uniformly small by taking $R'$ sufficiently large.
\end{definition}
\begin{assertion}
  \label{ass2}
There exist constants $C_1,C_2>0$ depending only on $F_1,F_2$ and complex
valued functions $B_1(z),B_2(z)$ in $D(\infty ,R')$ satisfying property $\diamondsuit $,
such that for all $z\in D(\infty ,R')$,
\begin{equation}
\label{ass2a}
|A_1(z)-B_1(z)|\leq \frac{C_1}{R'}(e^{|y|}-1),\qquad |A_2(z)-B_2(z)|\leq \frac{C_2}{R'}(1-e^{-|y|}).
\end{equation}
\end{assertion}

We come back to our study of the vertical projection of $E$.
If $y\geq 0$, then the right-hand side of (\ref{eq:n}) can be written as
$e^y\left( -\frac{ie^{ix}}{2}+D_1(z)\right) $, where
$D_1(z)=\frac{e^{-y}}{2}\left( e^{-y}ie^{ix}+\overline{A_1(z)}-A_2(z)\right) $.
In particular, the triangle inequality gives
\[
|D_1(z)|\leq \frac{e^{-y}}{2}\left( e^{-y}+\sum _{i=1}^2|A_i(z)-B_i(z)|+\sum _{i=1}^2|B_i(z)|\right)
\]
\begin{equation}
\label{eq:|D1|}
\stackrel{(\ref{ass2a})}{\leq }
\frac{e^{-y}}{2}\left( e^{-y}+ \frac{C_1}{R'}(e^{y}-1)+\frac{C_2}{R'}(1-e^{-y})+\sum _{i=1}^2|B_i(z)| \right)
=\frac{C_1}{2R'}+e^{-y}B_3(z),
\end{equation}
where $B_3(z)$ is a nonnegative function in $D(\infty ,R')$ that satisfies property $\diamondsuit$.
A similar analysis can be done for $y\leq 0$, concluding that
\begin{equation}
\label{eq:|x1+ix2|}
(x_1+ix_2)(z)=\left\{
\begin{array}{ll}
e^y\left( -\frac{ie^{ix}}{2}+D_1(z)\right) & \mbox{ if $y\geq 0$,}
\\
e^{-y}\left( \frac{ie^{ix}}{2}+D_2(z)\right) & \mbox{ if $y\leq 0$,}
\end{array}
\right.
\end{equation}
where $z=x+iy$ and $D_2(z)=\frac{e^{y}}{2}\left( -e^{y}ie^{ix}+\overline{A_1(z)}-A_2(z)\right) $; in particular,
\begin{equation}
\label{eq:|D2|}
|D_2(z)|\leq \frac{C_2}{2R'}+e^{y}B_4(z),
\end{equation}
where $B_4(z)$ is a nonnegative function in $D(\infty ,R')$ that satisfies property $\diamondsuit$.

Equation (\ref{eq:|x1+ix2|}) implies that for $n\in \N$ large, $|(x_1+ix_2)(z)|<2e^n$ whenever $|y|\leq n$,
from where $X(\C -\Delta (n))\subset C( 2e^n)$. Therefore,
choosing $R_E=2e^{n_0}$ with $n_0\in \N$ large, we deduce
from (\ref{eq:|x1+ix2|}) that $E-C(R_E)$ consists of two disjoint multivalued graphs
$\Sigma _1, \Sigma_2$ associated to smooth functions
$u_1,u_2\colon \widetilde{D}(\infty ,R_E)\to \R $ (recall that
$\wt{D}(\infty,R_E)$ denotes the universal cover of $D(\infty ,R_E)$), and
that are defined by the equations
\begin{equation}
\label{graph}
u_j(r,\t)\stackrel{(\ref{norm})}{=}x_3(z)=x+\l \log|z-\mu| ,
\end{equation}
where $x+iy=z=z(r,\t)$ satisfies $y\geq 0$ if $j=1$ (resp. $y\leq 0$ if $j=2$) and
\begin{equation}
\label{eq:reit}
re^{i\t }=\left\{
\begin{array}{ll}
e^y\left( -\frac{ie^{ix}}{2}+D_1(z)\right) & \mbox{ if $j=1$,}
\\
e^{-y}\left( \frac{ie^{ix}}{2}+D_2(z)\right) & \mbox{ if $j=2$.}
\end{array}
\right.
\end{equation}

In particular, $\Sigma _i$ is embedded, for $i=1,2$. Also,
from (\ref{eq:|D1|}), (\ref{eq:|D2|}) and (\ref{eq:reit})
we have that as functions of $r,\t $, the expressions of $x,y$ are of the form
\begin{equation}
\label{asym}
y=\log (2r)+\frac{b_1(r,\t )}{R'}, \qquad x=
\left\{
\begin{array}{ll}
\t +\frac{\pi }{2}+\frac{b_2(r,\t )}{R'}& \mbox{ if $j=1$,}
\\
\t -\frac{\pi }{2}+\frac{b_2(r,\t )}{R'}& \mbox{ if $j=2$,}
\end{array}
\right.
\end{equation}
where $b_1,b_2$ are bounded functions defined in $\widetilde{D}(\infty ,R_E)$.
Plugging (\ref{asym}) into (\ref{graph}) we deduce that
\begin{eqnarray}
u_1(r,\t )&=&
x+\l \log|z|+\l \log\left| 1-\frac{\mu }{z}\right| \nonumber
\\
&=&\t +\frac{\pi }{2}+\l \log \sqrt{\left( \t +\frac{\pi }{2}\right) ^2+[\log (2r)]^2}+\frac{b_3(r,\t )}{R'} \label{u1}
\\
&=& v_1(r,\t )+\frac{b_3(r,\t )}{R'},\nonumber
\end{eqnarray}
where $b_3$ is a bounded function defined in $\widetilde{D}(\infty ,R_E)$ and
$v_1$ is the multigraphing function defined in item~\ref{it6-2} of Theorem~\ref{main}.
Analogously, the multigraphing function that defines $\Sigma _2$ verifies that
$(u_2-v_2)R'$ is bounded.
Therefore, item~\ref{it6-2} of Theorem~\ref{main} is proved.

As for the property $\nabla u_i(r,\t )\to 0$ as $r\to \infty $ stated in
item~\ref{it6-1} of Theorem~\ref{main}, it follows from standard gradient estimates once
we know that around any point $re^{i\t }$ in $D(\infty ,R_E)$ with $r$ sufficiently large,
the boundary values of $u_i$ in a disk of radius $1$ centered at $re^{i\t }$ are
arbitrarily close to a constant value. This finishes the proof of item~\ref{it6-1} of Theorem~\ref{main}.

By item~\ref{it5} and our translation normalization,  if $a=0$
then the rays $r_T$ and $r_B$ are contained in the $x_3$-axis.
By item~\ref{it6-2}, $|u_i-v_i|$ can be made arbitrarily
small if $r$ is sufficiently large. Thus, in order to prove part I of item~\ref{it6-3},
it suffices to show that $|u_i-v_i|$ tends to zero when $r$ is bounded and $|\t |\to \infty $.
Consider the multigraph $E'$ over $D(\infty ,R_E)$ associated to the
function $v_i$ given by (\ref{ss}). The sequence of surfaces $E'-(0,0,2\pi n +\l \log n)$
can be seen to converge to the multigraph associated to the function
$v_{\infty }(r,\t )=\t +\frac{\pi }{2}-\l \log (2\pi )$, which is contained in a vertical
helicoid $H'$ whose the axis is the $x_3$-axis.
As by item~\ref{it7} the translated multigraphs $\Sigma _1-(0,0,2\pi n +\l \log n)$
converge to a multigraph contained in the vertical helicoid $H_T$ (which also contains
$x_3$-axis since $H_T$ contains $r_T$), then to finish part I of item~\ref{it6-2} it only
remains to show that $H=H'$. This equality holds since $|u_i-v_i|(r,\t )$ tends to zero
as $r\to \infty $. Now part~I of item~\ref{it6-2} is proved.

As for part II of item~\ref{it6-2} and regardless of the condition $a=0$, the facts that
the separation function between $v_1$ and $v_2$ has an asymptotic
value of $\pi$ as $r\to \infty $ together with the estimates in item~\ref{it6-2}
imply that the separation function $w(r,\theta )=u_1(r,\theta )-u_2(r,\theta )$
between $u_1$ and $u_2$ converges to $\pi $ as $r\to \infty $. If $r$ is bounded and
$|\t |\to \infty $, then the same limiting property for $w(r,\theta )$ holds by item~\ref{it7},
since the separation of the multigraphs in the helicoid
$H_B$ or in $H_T$ is constant $\pi $.
Now the proof of item~\ref{it6} is complete.

Next we prove item~\ref{it8}. Suppose that $X_2\colon D(\infty,
R_2)\to \rth$ is another conformal minimal immersion with the same
flux vector $(a,0,-2\pi \l )$ as $X$ and with Weierstrass data
$(g_2,dh_2)$ as described in item~\ref{it8}. Since $X,X_2$ have the
same flux vector, then item~\ref{it5} allows us to find translations
$Y_1=E+\tau_1$ and $Y_2=X_2(D(\infty, R_2))+\tau_2$ of $E=X(D(\infty, R_1))$ and
$X_2(D(\infty, R_2))$ respectively, such that
$Y_1, Y_2$ each has the same half-axes projections
$(x_T,y_T)=(0,-\frac{a}{4})$, $(x_B,y_B)=(0,\frac{a}{4})$. Now translate vertically
$Y_1,Y_2$ so that their third coordinate functions are respectively given by
\[
x_3(z)=x+\l \log |z-\mu |,\quad y_3(z)=x+\l \log |z-\mu _2|.
\]
Thus, we can apply item~\ref{it6-2} to $Y_1,Y_2$ and the corresponding functions $v_1,v_2$
in~(\ref{ss}), (\ref{sss}) are the same for both $Y_1,Y_2$. By transitivity
we have that given any $n\in \N$, there is a large
$R_n>0$ such that each of the two multivalued graphs of $Y_1 - C(R_n)$ is
$\frac{2}{n}$-close to the corresponding multivalued graph of $Y_2 -
C(R_n)$ in $\widetilde{D}(\infty ,R_n)$.
Also, the proof of item~\ref{it5} can be adapted to show that both $Y_1-(0,0,2\pi n+\l \log n)$,
$Y_2-(0,0,2\pi n+\l \log n)$ both converge to the same vertical helicoid. Hence, $Y_1$ and
$Y_2$ are asymptotic in $\rth$, and item~\ref{it8} is proved.

Item~\ref{it9} follows from the asymptotic description of an example $E$
with flux vector $(a,0,-b)$ given in \ref{it6} in terms of the flux components
$a$, $b=2\pi \l $.
The symmetry property in item~\ref{it10} follows directly from Assertion~\ref{ass6.6}.
Finally, item~\ref{it1} on the embeddedness of some subend of $E$ follows immediately
from the asymptotic embeddedness information described in items \ref{it7} and~\ref{it6},
together with the next straightforward observation: Embeddedness of $E$ outside of a vertical
cylinder follows from the sentence just after equation~(\ref{eq:reit}) together with the
positivity of the separation function given in part~II of item~\ref{it6-3}, while
embeddedness inside a vertical cylinder and outside of a ball follows from~item~\ref{it7}.
\end{proof}

\section{Appendix.}
\label{appendix}
{\it Proof of Assertion~\ref{ass2}.}
We will denote by $z_0$ the (common) base point for the integration in the functions $A_i(z)$
(recall that Assertion~\ref{ass2} is only used in the proof of item~\ref{it6} of Theorem~\ref{main},
whose first sentence assumes that $E$ has been translated as required in (\ref{norm})).
We decompose the integration path in $A_i(z)$ into three consecutive embedded arcs joined
by their common end points: an arc $\G $ of the circle of radius $\{ |z|=|z_0|\} $ from
$z_0$ to a point $R_0$ in the positive or negative $x$-axis (we take $R_0$ with the same sign as
$\mbox{Re}(z)$), and two segments parallel to the $x$- and $y$-axes. This decomposition procedure works
if $|\mbox{Re}(z)|\geq |z_0|$; if on the contrary $|\mbox{Re}(z)|<|z_0|$, then the decomposition can be
changed in a straightforward manner taking $R_0$ in the $y$-axis, which does not affect essentially
the arguments that follow.

Thus,
\begin{eqnarray}
A_1(z)=\int _{\G }e^{-i\xi}F_1(\xi )\, d\xi +\int _{R_0}^xe^{-iu}F_1(u)\, du +ie^{-ix}\int _0^ye^{v}F_1(x+iv)\, dv, \label{ass2a1}
\\
A_2(z)=\int _{\G }e^{i\xi}F_2(\xi )\, d\xi +\int _{R_0}^xe^{iu}F_2(u)\, du +ie^{ix}\int _0^ye^{-v}F_2(x+iv)\, dv. \label{ass2a2}
\end{eqnarray}
We define $B_1(z)$ (resp. $B_2(z)$) as the sum of the two first integrals in (\ref{ass2a1})
(resp. in (\ref{ass2a2})). It remains to show that both $B_1,B_2$ satisfy property
$\diamondsuit $, and that the modulus of the third integral in (\ref{ass2a1}) (resp. in (\ref{ass2a2}))
satisfies the first inequality in (\ref{ass2a}) (resp. the second inequality).

As $F_1,F_2$ are holomorphic and vanish at $\infty $, then $F_i(z)=\frac{C_i'}{z}+\frac{G_i(z)}{z^2}$ in
$D(\infty ,R')$, where $C_i'\in \C $ and $G_i$ is holomorphic in $D(\infty ,R')\cup \{ \infty \} $.
In particular, the first integral of each of the right-hand sides
of (\ref{ass2a1}) and (\ref{ass2a2}) define a complex constant (independent of $z$ such that $|\mbox{Re}(z)|\geq |z_0|$)
that can be taken arbitrarily small if $R'$ is sufficiently large. As for
the second integral of (\ref{ass2a1}), (\ref{ass2a2}),
\begin{eqnarray}
\int _{R_0}^xe^{-iu}F_1(u)\, du =C_1'\int _{R_0}^x\frac{e^{-iu}}{u}\, du +\int _{R_0}^x\frac{e^{-iu}}{u^2}G_1(u)\, du, \label{ass2b1}
\\
\int _{R_0}^xe^{iu}F_2(u)\, du =C_2'\int _{R_0}^x\frac{e^{iu}}{u}\, du +\int _{R_0}^x\frac{e^{iu}}{u^2}G_2(u)\, du. \label{ass2b2}
\end{eqnarray}
To study the first integral in the right-hand side of (\ref{ass2b1}), (\ref{ass2b2}), we decompose the real and imaginary parts
of each of these integrals
as a sum of integrals along consecutive intervals of length $\pi $
where the integrand has constant sign, and apply the
alternating series test to conclude that the complex valued functions
\[
\int _{R_0}^x\frac{e^{-iu}}{u}\, du ,\quad \int _{R_0}^x\frac{e^{iu}}{u}\, du
\]
satisfy property $\diamondsuit$. Regarding the second integrals in the
right-hand side of (\ref{ass2b1}), (\ref{ass2b2}), the fact that $G_i$ is bounded and that
$\left| \int _{R_0}^x\frac{du}{u^2}\right| =\left| \frac{1}{R_0}-\frac{1}{x}\right| $ imply that
the second integrals in the right-hand side of (\ref{ass2b1}), (\ref{ass2b2})
also satisfy property $\diamondsuit$. Thus we conclude that the functions $B_1(z),B_2(z)$
satisfy property $\diamondsuit$.

We next study the modulus of the third integrals on the right-hand sides of (\ref{ass2a1}) and (\ref{ass2a2}).
Using again that $F_i$ vanishes at $\infty $ for $i=1,2$, we find $C_i>0$
depending only on $F_i$ such that $|F_i(z)|\leq C_i|z|^{-1}$ whenever $|z|\geq R'$. Hence,
\[
\begin{array}{r}
{\displaystyle \left| \int _0^ye^{v}F_1(x+iv)\, dv\right| \leq \frac{C_1}{R'}\int _0^{|y|}e^v\, dv
=\frac{C_1}{R'}(e^{|y|}-1), }
\\
{\displaystyle
\left| \int _0^ye^{-v}F_2(x+iv)\, dv\right| \leq \frac{C_2}{R'}(1-e^{-|y|}),
}
\end{array}
\]
which finishes the proof of Assertion~\ref{ass2}.
{\hfill\penalty10000\raisebox{-.09em}{$\Box$}\par\medskip}

\center{William H. Meeks, III at profmeeks@gmail.com\\
Mathematics Department, University of Massachusetts, Amherst, MA 01003}
\center{Joaqu\'\i n P\'{e}rez at jperez@ugr.es\\
Department of Geometry and Topology and Institute of Mathematics (IEMath-GR),
University of Granada, 18071 Granada, Spain}

\bibliographystyle{plain}
\bibliography{bill}

\begin{thebibliography}{10}

\bibitem{bb2}
J.~Bernstein and C.~Breiner.
\newblock Conformal structure of minimal surfaces with finite topology.
\newblock {\em Comm. Math. Helv.}, 86(2):353--381, 2011.
\newblock MR2775132, Zbl 1213.53011.

\bibitem{bb1}
J.~Bernstein and C.~Breiner.
\newblock Helicoid-like minimal disks and uniqueness.
\newblock {\em J. Reine Angew. Math.}, 655:129--146, 2011.
\newblock MR2806108, Zbl 1225.53008.

\bibitem{cm34}
T.~H. Colding and W.~P. Minicozzi~II.
\newblock An excursion into geometric analysis.
\newblock In {\em Surveys of Differential Geometry IX - Eigenvalues of
  Laplacian and other geometric operators}, pages 83--146. International Press,
  edited by Alexander Grigor'yan and Shing Tung Yau, 2004.
\newblock MR2195407, Zbl 1076.53001.

\bibitem{cm21}
T.~H. Colding and W.~P. Minicozzi~II.
\newblock The space of embedded minimal surfaces of fixed genus in a
  $3$-manifold {I}; {E}stimates off the axis for disks.
\newblock {\em Ann. of Math.}, 160:27--68, 2004.
\newblock MR2119717, Zbl 1070.53031.

\bibitem{cm22}
T.~H. Colding and W.~P. Minicozzi~II.
\newblock The space of embedded minimal surfaces of fixed genus in a
  $3$-manifold {I}{I}; {M}ulti-valued graphs in disks.
\newblock {\em Ann. of Math.}, 160:69--92, 2004.
\newblock MR2119718, Zbl 1070.53032.

\bibitem{cm24}
T.~H. Colding and W.~P. Minicozzi~II.
\newblock The space of embedded minimal surfaces of fixed genus in a
  $3$-manifold {I}{I}{I}; {P}lanar domains.
\newblock {\em Ann. of Math.}, 160:523--572, 2004.
\newblock MR2123932, Zbl 1076.53068.

\bibitem{cm23}
T.~H. Colding and W.~P. Minicozzi~II.
\newblock The space of embedded minimal surfaces of fixed genus in a
  $3$-manifold {I}{V}; {L}ocally simply-connected.
\newblock {\em Ann. of Math.}, 160:573--615, 2004.
\newblock MR2123933, Zbl 1076.53069.

\bibitem{cm35}
T.~H. Colding and W.~P. Minicozzi~II.
\newblock The {C}alabi-{Y}au conjectures for embedded surfaces.
\newblock {\em Ann. of Math.}, 167:211--243, 2008.
\newblock MR2373154, Zbl 1142.53012.

\bibitem{cm25}
T.~H. Colding and W.~P. Minicozzi~II.
\newblock The space of embedded minimal surfaces of fixed genus in a
  $3$-manifold {V}; {F}ixed genus.
\newblock {\em Ann. of Math.}, 181(1):1--153, 2015.
\newblock MR3272923, Zbl 06383661.

\bibitem{col1}
P.~Collin.
\newblock Topologie et courbure des surfaces minimales de $\rth$.
\newblock {\em Ann. of Math. (2)}, 145--1:1--31, 1997.
\newblock MR1432035, Zbl 886.53008.

\bibitem{hkp1}
L.~Hauswirth, J.~P\'{e}rez, and P.~Romon.
\newblock Embedded minimal ends of finite type.
\newblock {\em Transactions of the AMS}, 353:1335--1370, 2001.
\newblock MR1806738, Zbl 0986.53005.

\bibitem{mpe5}
W.~H. Meeks~III and J.~P\'{e}rez.
\newblock Finite type annular ends for harmonic functions.
\newblock Preprint at http://arXiv.org/abs/0909.1963.

\bibitem{mpe2}
W.~H. Meeks~III and J.~P\'{e}rez.
\newblock The classical theory of minimal surfaces.
\newblock {\em Bulletin of the AMS}, 48:325--407, 2011.
\newblock MR2801776, Zbl 1232.53003.

\bibitem{mpe10}
W.~H. Meeks~III and J.~P\'{e}rez.
\newblock {\em A survey on classical minimal surface theory}, volume~60 of {\em
  University Lecture Series}.
\newblock AMS, 2012.
\newblock ISBN: 978-0-8218-6912-3; MR3012474, Zbl 1262.53002.

\bibitem{mpr9}
W.~H. Meeks~III, J.~P\'{e}rez, and A.~Ros.
\newblock The embedded {C}alabi-{Y}au conjectures for finite genus.
\newblock Work in progress.

\bibitem{mpr14}
W.~H. Meeks~III, J.~P\'{e}rez, and A.~Ros.
\newblock The local picture theorem on the scale of topology.
\newblock Preprint at http://arxiv.org/abs/1505.06761.

\bibitem{mpr10}
W.~H. Meeks~III, J.~P\'{e}rez, and A.~Ros.
\newblock Local removable singularity theorems for minimal laminations.
\newblock To appear in J. Differential Geom. Preprint at
  http://arxiv.org/abs/1308.6439.

\bibitem{mpr11}
W.~H. Meeks~III, J.~P\'{e}rez, and A.~Ros.
\newblock Structure theorems for singular minimal laminations.
\newblock Preprint available at
  http://wdb.ugr.es/local/jperez/publications-by-joaquin-perez/.

\bibitem{mr8}
W.~H. Meeks~III and H.~Rosenberg.
\newblock The uniqueness of the helicoid.
\newblock {\em Ann. of Math.}, 161:723--754, 2005.
\newblock MR2153399, Zbl 1102.53005.

\bibitem{os3}
R.~Osserman.
\newblock Global properties of minimal surfaces in {$E^3$} and {$E^n$}.
\newblock {\em Ann. of Math.}, 80(2):340--364, 1964.
\newblock MR0179701, Zbl 0134.38502.

\bibitem{rose1}
H.~Rosenberg.
\newblock Minimal surfaces of finite type.
\newblock {\em Bull. Soc. Math. France}, 123:351--354, 1995.
\newblock MR1373739, Zbl 0848.53004.

\end{thebibliography}

\end{document}